\documentclass[10pt,reqno]{amsart}
\usepackage{fullpage}
\usepackage{amsthm,amsmath,amsfonts,amssymb,euscript,hyperref,graphics,color,slashed,mathrsfs,tikz,wrapfig,float,caption,enumerate,graphicx,mathtools,extarrows,cite}
\allowdisplaybreaks

\hypersetup{colorlinks, citecolor=blue, filecolor=blue, linkcolor=blue, urlcolor=blue}

\numberwithin{equation}{section}

\newtheorem{Theorem}{Theorem}[section]
\newtheorem{Lemma}[Theorem]{Lemma}

\newtheorem{Definition}[Theorem]{Definition}

\newtheorem{Property}[Theorem]{Property}

\title[Monge-Amp\`ere]{Existence, uniqueness and interior regularity of\\ viscosity solutions for a class of Monge-Amp\`ere type equations}

\author[Mengni Li]{}
\address{Mengni Li, School of Mathematics, Southeast University, Nanjing, Jiangsu, China}
\email{krisymengni@163.com}
\author[You Li]{}
\address{You Li, School of Mathematics and Computational Sciences, Xiangtan University, Hunan, China}
\email{thu3141@126.com}

\subjclass[2010]{}


\begin{document}

\maketitle

\centerline{Mengni Li and You Li\footnote{Corresponding author: You Li}}

\begin{abstract}
The	Monge-Amp\`ere type equations over bounded convex domains arise in a host of geometric applications. In this paper, we focus on the Dirichlet problem for a class of Monge-Amp\`ere type equations, which can be degenerate or singular near the boundary of convex domains. Viscosity subsolutions and viscosity  supersolutions to the problem can be constructed via comparison principle. Finally, we demonstrate the existence, uniqueness and interior regularity (including $W^{2,p}$  with $p\in(1,+\infty)$, $C^{1,\mu}$ with $\mu\in(0,1)$, and $C^\infty$) of the viscosity solution to the problem.

\noindent{\bf \\ Running head: } 
 Monge-Amp\`ere type equation
\noindent{\bf \\ 2020 Mathematics Subject Classification: }35J96, 35D40, 35A01, 35A02, 35B65
\end{abstract}

\section{Introduction}\label{section: introduction}
This paper is devoted to the study of the following Dirichlet problem
\begin{align*}
\det D^2 u&=f(x,u,\nabla u)\text{\hspace{0.2cm} in } \Omega,\stepcounter{equation}\tag{\theequation}\label{equation}\\
u&=0\text{\hspace{0.2cm} on } \partial\Omega,\stepcounter{equation}\tag{\theequation}\label{bcondition}
\end{align*}
where $\Omega\subseteq \mathbb{R}^n(n \geqslant 2)$ is a bounded convex domain, $u:\Omega\to\mathbb{R}$ is a convex function, and $f:\Omega	\times(-\infty,0) \times \mathbb{R}^n\to\mathbb{R}$ satisfies  the following structure conditions:
\begin{enumerate}
\item[$(f_1)$] $f(x,z,q)\in C(\Omega\times (-\infty,0)\times \mathbb{R}^n)$;
\item[$(f_2)$] $\text{for all } 
x\in\Omega, q\in \mathbb{R}^n, f(x,z,q)\text{ is non-decreasing in the second argument }z$;
\item[$(f_3)$] for some constants $A > 0$, $\alpha\in\mathbb{R}$, $\beta\geqslant n+1$, $\gamma<\min\{n+\alpha,\beta-n+1\}$,  there holds
\[0< f(x,z,q)\leqslant
Ad_x^{\beta-n-1}|z|^{-\alpha}(1+|q|^2)^{\frac{\gamma}{2}},\ \ \forall
(x,z,q)\in\Omega	\times(-\infty,0) \times \mathbb{R}^n,\]
where $d_x=\operatorname{dist}(x,\partial\Omega)$.  
\end{enumerate}

Here, 
$(f_1)$ 
is a  regularity assumption, 
$(f_2)$ 
is a monotone assumption, and 
$(f_3)$ is a boundness assumption. 
We observe that $u$, $|\nabla u|$ and $\det D^2u$  are invariant under translation and rotation transformations, which gives that  the equation \eqref{equation} will still satisfy the structure condition $(f_3)$  after  translation and rotation transformations. It is also worth pointing out that the equation \eqref{equation} can  be degenerate or singular near the boundary $\partial\Omega$ based on the structure condition $(f_3)$.  This behavior can be understood in the sense that the right-hand-side function  $f(x,u,\nabla u)$ degenerates near $\partial\Omega$ when $\beta$ is taken sufficiently large, and it may become singular near $\partial\Omega$ when $\alpha$ is sufficiently large as a result of $u=0$ there. 

Throughout the literature, the equation \eqref{equation} is closely related to an abundance of geometric problems. Two large types of the right-hand-side function in \eqref{equation} with representative examples might appear as follows:  when $$f(x,z,q)=g(x)(1+|q|^2)^{\frac{\gamma}{2}},$$  the equation \eqref{equation} includes the prescribed Gauss curvature equation as a special case if taking   $\gamma=n+2$ (see  \cite{Trudinger-Urbas83}); when  $$f(x,z,q)=g(x)|z|^{-\alpha},$$  the equation \eqref{equation} can arise from the $L_p$-Minkowski problem if denoting $\alpha=1-p$  (see \cite{Lutwak,Chen-Li-Zhu}).  Therein lies the critical exponent case of the latter example with  $p=-n-1$, namely the centroaffine Minkowski problem (see \cite{Chou-Wang,Jian-Lu-Zhu}), which is also of great and independent interest. In particular, for the case $f(x,z,q)=|z|^{-(n+2)},$ if $u$ is a solution to the problem \eqref{equation}-\eqref{bcondition}, then $-(1/u)\sum u_{x_ix_j}dx_idx_j$ provides the Hilbert metric in convex domain $\Omega$ (see \cite{Loewner-Nirenberg}), and the Legendre transform of  $u$ defines a complete hyperbolic affine sphere (see \cite{Cheng-Yau77}).  We also refer the readers to \cite{Figalli,Trudinger-Wang08adv}    for more information  and applications.

The study of  Monge-Amp\`ere type equations 
is a long-standing topic. In the past decades,  
 plenty of works have been taken to investigate the existence, uniqueness and interior regularity  of viscosity solutions to the problem \eqref{equation}-\eqref{bcondition} over convex domains.  Let us quote here some representative contributions to motivate this paper. 
 The history goes back at least to the work of  Cheng and Yau \cite{Cheng-Yau77}, where the special case was considered when the right-hand-side $f(x,u,\nabla u)$ is independent of $\nabla u$.  They proved that if $\Omega\subset\mathbb{R}^n$ is a bounded $C^2$ strictly convex domain, $f(x,z)$ is a positive $C^k$ ($k\geqslant 3$) function satisfying $(f_2)$ and 
some structure condition (similar to $(f_3)$), then the problem \eqref{equation}-\eqref{bcondition} admits a unique convex solution $u\in C^{k+1,\delta}(\Omega)$ for any $\delta\in(0,1)$. 
 A bit later,  Caffarelli, Nirenberg and Spruck \cite{Caffarelli-Nirenberg-Spruck} contributed to the general case by showing 
 that if  $\Omega\subset\mathbb{R}^n$ is a bounded $C^\infty$ strictly convex domain, $f(x,z,q)$ is a positive $C^\infty$ function   with $f_z\geqslant 0$, and there exists a globally smooth convex subsolution $\underline{u}\in C^2(\overline{\Omega})$, then the problem \eqref{equation}-\eqref{bcondition} admits a unique strictly convex solution $u\in C^{\infty}(\overline{\Omega})$ with $u\geqslant \underline{u}$. 
 Trudinger and Urbas \cite{Trudinger-Urbas83} subsequently  demonstrated that the problem \eqref{equation}-\eqref{bcondition} is uniquely solvable with convex solution $u\in C^2(\Omega)\cap C^{0,1}(\overline{\Omega})$ under the assumption that $\Omega\subset\mathbb{R}^n$ is a $C^{1,1}$ uniformly convex domain, $f(x,z,q)$ is a positive $C^{1,1}$ function satisfying $f_z\geqslant 0$, $(f_3)$ near $\partial\Omega$ with $\alpha=0$ and  $\gamma\leqslant\beta$, and 
 $ f(x,-N,q)\leqslant g(x)/h(q)$  for all  $x\in
 \Omega$  and $q\in \mathbb{R}^n$, where $N\in\mathbb{R},\ g\in L^1({\Omega}),\ h\in L^1_{\operatorname{loc}}(\mathbb{R}^n)$ are positive with $\int_{\Omega}g<\int_{\mathbb{R}^n}h.$ 
 Considering $\Omega\subset\mathbb{R}^n$ a uniformly convex domain, Urbas  \cite{Urbas88} then extended 
 these results to the following type of equation
 $\det D^2u=g(x,u,Du)/h(Du)$,
 and proved that there is a unique  convex solution $u\in C^2(\Omega)$ if  $g\in C^{1,1}(\Omega\times\mathbb{R}\times\mathbb{R}^n),\ h\in C^{1,1}(\mathbb{R}^n)\cap L^1(\mathbb{R}^n)$ satisfy 
  some structure conditions (see Theorem 4.10 in \cite{Urbas88}). 
 
 The interested readers can also consult \cite{Figalli,Le,Tso,Jiang-Trudinger-Yang,Tian,Liu-Wang,Jian-Wang1,Trudinger-Wang08adv,Hong-Huang-Wang,Guan-Ma,Guanbo} 
and the references therein to find more results and methods on this topic.  
However,  among these studies, we see that the existence, uniqueness and regularity theories on the problem \eqref{equation}-\eqref{bcondition} with general right-hand-side $f(x,u,\nabla u)$ are still far from complete. In fact, 
the principal purpose of this paper is to make a further step in this direction. 

Furthermore, the authors \cite{Li-Li} have recently provided the global H\"older regularity for the problem \eqref{equation}-\eqref{bcondition} under the assumption that there exists a convex solution to this problem.  These regularity  estimates can indeed be viewed as the \textit{a priori} estimates when we prove the existence of viscosity solutions (see Definition \ref{def:vis}). 
Our first main theorem  can be stated as follows: 
\begin{Theorem}[Existence Theorem]\label{thm1}
	Suppose $\Omega\subset\mathbb{R}^n$ is a bounded convex domain, $f(x,z,q)$ satisfies $(f_1)$, $(f_2)$ and $(f_3)$. Then the problem \eqref{equation}-\eqref{bcondition} admits a   viscosity solution over $\Omega$. 
\end{Theorem}

Based on the existence of viscosity solutions, here comes a natural question whether we can show the uniqueness and interior regularity of viscosity solutions. In fact, under the same assumptions of Theorem \ref{thm1}, there hold a series of interior regularity results as in the following second main theorem:
\begin{Theorem}[Interior Regularity Theorem]\label{thm2}
	Under the conditions of Theorem \ref{thm1}, if $u$ is a viscosity solution to the problem \eqref{equation}-\eqref{bcondition} over $\Omega$, then the following two conclusions hold:
	\begin{enumerate}[(i)]
		\item For any $\Omega_0\subset\subset\Omega$, we have that $u$ is strictly convex in $\Omega_0$ and 
		\begin{align*}
			&u\in W^{2,p}(\overline{\Omega_0}),\ \ \forall p\in (1,+\infty),\\
			&u\in C^{1,\mu}(\overline{\Omega_0}),\ \ \forall \mu\in (0,1).
		\end{align*}
		
		\item If $f(x,z,q)\in C^\infty(\Omega\times (-\infty,0)\times \mathbb{R}^n)$, then \[u\in C^\infty(\Omega).\]
	\end{enumerate} 
\end{Theorem}

Our third main theorem is about the uniqueness issue. Though we believe that the viscosity solution constructed in Theorem \ref{thm1} is unique, our  methods are limited at present and hence we have to leave this property as an open problem for further studies. Instead, we are capable of showing 
 the uniqueness of the  viscosity solution under the assumptions of Theorem \ref{thm1} together with the following two extra assumptions: 
 \begin{enumerate}
 	\item[$(f_1')$] there exists a constant $\delta_0\in (0,1)$ such that 
 	$f(x,z,q)\in C_{\operatorname{loc}}^{0,\delta_0}(\Omega\times(-\infty,0)\times\mathbb{R}^n)$;
 	\item[$(f_2')$] for all $
 	x\in\Omega, q\in \mathbb{R}^n$, $f(x,z,q)$ is strictly increasing in the second argument $z$. 
 \end{enumerate}
  More precisely, we can elaborate our third main theorem as follows:
		
\begin{Theorem}[Uniqueness Theorem]\label{thm3}
	Under the conditions of Theorem \ref{thm1}, if $f(x,z,q)$ further satisfies $(f_1')$ and $(f_2')$, 
	then the viscosity solution to  the problem \eqref{equation}-\eqref{bcondition} over $\Omega$ is unique. 
\end{Theorem}

The remainder of this paper is arranged as follows.  In Section \ref{Sec:Pre}, we present basic concepts and detailed preliminary lemmas for the problem \eqref{equation}-\eqref{bcondition}. We devote Section \ref{Sec:W} and Section \ref{Sec:u} to constructing  auxiliary functions $W$ and $u$ respectively, which are intended to motivate the investigation of viscosity solution in what follows. In Section  \ref{Sec:EV}, we 
 show the existence  of viscosity solutions to the problem \eqref{equation}-\eqref{bcondition}. In Section \ref{Sec:IR}, 
based on the existence result, we prove Theorem \ref{thm2} and study a series of interior regularity, including $W^{2,p}$  with $p\in(1,+\infty)$,  $C^{1,\mu}$  with $\mu\in(0,1)$, and  $C^\infty$ of viscosity solutions. 
The uniqueness  of viscosity solution and the proof of our third main Theorem \ref{thm3}  are finally given in 
Section \ref{Sec:Unique}. The underlying idea of our proof is to use comparison principles appropriately, to exploit the convex analysis for domains and functions subtly,  as well as to  construct viscosity subsolutions and viscosity supersolutions carefully.

\section{Preliminaries}\label{Sec:Pre}

Let us begin with the concepts of classical solutions and viscosity solutions to the problem \eqref{equation}-\eqref{bcondition}.

\begin{Definition}\label{def:class}
		Let $u\in C^2(\Omega)\cap C(\overline{\Omega})$ be a convex function. 
	\begin{enumerate}[(i)]
		\item 
	We say that $u$ is  a classical subsolution to \eqref{equation} over $\Omega$ if $$\operatorname{det} D^2 u\geqslant f(x,u,\nabla u) \  \text{ in } \Omega.$$  
	\item 	We say that $u$ is  a classical supersolution to \eqref{equation}  over $\Omega$ if $$\operatorname{det} D^2 u\leqslant f(x,u,\nabla u) \  \text{ in } \Omega.$$ 
	\item $u$ is called a classical solution to \eqref{equation}  over $\Omega$ if it is both a classical subsolution to \eqref{equation}  over $\Omega$ and a classical supersolution to \eqref{equation}  over $\Omega$.
	\item $u$ is called a classical solution to the  problem  \eqref{equation}-\eqref{bcondition}  over $\Omega$ if it is a classical solution to \eqref{equation}  over $\Omega$ and satisfies \eqref{bcondition} on $\partial\Omega$.  
\end{enumerate}
\end{Definition}

\begin{Definition}\label{def:vis}
	Let $u\in C(\overline{\Omega})$ be a convex function. 
\begin{enumerate}[(i)]
\item We say that $u$ is a viscosity subsolution to \eqref{equation}  over $\Omega$ 
if 
for any point $x_0\in\Omega$,  any open neighborhood $U(x_0)$ and any convex function $\phi\in C^2(U(x_0)\cap\Omega)$ satisfying
\[(\phi-u)(x)\geqslant (\phi-u)(x_0),\ \ \forall x\in U(x_0)\cap\Omega,\]
we have
\[\det D^2\phi(x_0)\geqslant    f(x_0,u(x_0),\nabla \phi(x_0)).\]
\item We say that $u$ is a viscosity  supersolution to \eqref{equation}  over $\Omega$ 
if 
for any point $x_0\in\Omega$,  any open neighborhood $U(x_0)$ and any convex function $\phi\in C^2(U(x_0)\cap\Omega)$ satisfying
\[(\phi-u)(x) \leqslant (\phi-u)(x_0),\ \ \forall x\in U(x_0)\cap\Omega,\]
we have
\[\det D^2\phi(x_0) \leqslant f(x_0,u(x_0),\nabla \phi(x_0)).\]
\item  $u$ is called a viscosity solution to \eqref{equation}  over $\Omega$ if it is both a viscosity subsolution to \eqref{equation}  over $\Omega$ and a viscosity supersolution to \eqref{equation}  over $\Omega$.
	\item $u$ is called a viscosity  solution to the   problem  \eqref{equation}-\eqref{bcondition}  over $\Omega$ if it is a viscosity solution to \eqref{equation}  over $\Omega$ and satisfies \eqref{bcondition} on $\partial\Omega$. 
\end{enumerate}
\end{Definition}

We also refer the readers to the references \cite{Guti,Crandall} for example for more details on viscosity solutions. We now turn to prove an equivalent condition for viscosity solutions and a comparison principle for viscosity subsolutions/supersolutions.

\begin{Lemma}\label{lemma:equiv}
	If $u\in C^2(\Omega)$, then the following equivalence holds:
	\begin{equation}\label{equivalence1}
		\text{$u$ is a viscosity solution to \eqref{equation}  over $\Omega$ if and only if $u$ is a classical solution to \eqref{equation}  over $\Omega$.}
	\end{equation}
\end{Lemma}
\begin{proof}
	Let $u\in C^2(\Omega)$. 
To prove \eqref{equivalence1}, we only need to prove that both \eqref{equivalence2} and \eqref{equivalence3} hold:
\begin{equation}\label{equivalence2}
	\text{$u$ is a viscosity subsolution to \eqref{equation}  over $\Omega$ if and only if $u$ is a classical subsolution to \eqref{equation}  over $\Omega$.}
\end{equation}
\begin{equation}\label{equivalence3}	\text{$u$ is a viscosity supersolution to \eqref{equation}  over $\Omega$ if and only if $u$ is a classical supersolution to \eqref{equation}  over $\Omega$.}
\end{equation}
Now the proof of this lemma can be divided into the following two steps.

\smallskip

\noindent
\textbf{Step 1}: We show that if $u$ is a viscosity subsolution (supersolution) to \eqref{equation}  over $\Omega$, then $u$ is a classical subsolution  (supersolution) to \eqref{equation}  over $\Omega$.

Let $u$ be a viscosity subsolution  (supersolution) to \eqref{equation}  over $\Omega$. For any $x_0\in\Omega$ and any $x\in U(x_0)\cap \Omega$,  we take $\phi(x)=u(x)$. Since $u\in C^2(\Omega)$, there holds $\phi\in C^2(U(x_0)\cap \Omega)$. As a result of $\phi(x)=u(x)$, 
we have $\phi(x)-u(x)=0=\phi(x_0)-u(x_0)$ and therefore 
\[(\phi-u)(x)\geqslant(\leqslant)(\phi-u)(x_0).\]
By Definition \ref{def:vis}, it follows that 
\[\det D^2\phi(x_0)\geqslant  (\leqslant) f(x_0,u(x_0),\nabla \phi(x_0)),\]
which together with $\phi(x)=u(x)$ implies 
\[\det D^2u(x_0)\geqslant  (\leqslant) f(x_0,u(x_0),\nabla u(x_0)).\]
Since $x_0\in\Omega$ is arbitrary, we infer that $u$ is a classical subsolution (supersolution) to \eqref{equation}  over $\Omega$.

\smallskip

\noindent
\textbf{Step 2}: We show that if $u$ is a classical subsolution (supersolution) to \eqref{equation}  over $\Omega$, then $u$ is a viscosity subsolution (supersolution) to \eqref{equation}  over $\Omega$.

Let $u$ be a classical subsolution  (supersolution) to \eqref{equation}  over $\Omega$. Then for any $x_0\in\Omega$, we have 
\[\det D^2u(x_0)\geqslant  (\leqslant) f(x_0,u(x_0),\nabla u(x_0)).\]
For any convex function $\phi\in C^2(U(x_0)\cap\Omega)$ satisfying
\[(\phi-u)(x)\geqslant(\leqslant)(\phi-u)(x_0),\ \ \forall x\in U(x_0)\cap\Omega,\]
it is clear that $\phi-u$ takes its minimum (maximum) at $x_0$. We immediately obtain
$$\nabla(\phi-u)\big|_{x_0}=0,\ \ \nabla^2(\phi-u)\big|_{x_0}\geqslant(\leqslant)0,$$
i.e. 
$$\nabla \phi(x_0)=\nabla u(x_0),\ \ \det D^2\phi(x_0)\geqslant(\leqslant)\det D^2u(x_0).$$
We remark here that in the sequel, for any function $\Phi$, we always use the notation $\nabla^2\Phi\geqslant(\leqslant)0$ to represent that the Hessian matrix $(\nabla^2\Phi)$ is positive (negative) semi-definite. 
It then follows that 
\[\det D^2\phi(x_0)\geqslant(\leqslant)\det D^2u(x_0)\geqslant  (\leqslant)f(x_0,u(x_0),\nabla u(x_0))= f(x_0,u(x_0),\nabla \phi(x_0)),\]
which means that $u$ is a  viscosity subsolution (supersolution) to \eqref{equation}  over $\Omega$.

Until now,  we have finished the proof of this lemma. 
\end{proof}

For the readers' convenience, we now collect Lemma 2.1 in \cite{J-WXD} as the following  comparison principle. 
This version will provide more convenience for studying the existence of viscosity solutions  in this paper, see Section \ref{Sec:EV} for more details. 
 
\begin{Lemma}[comparison principle of Dirichlet problem  \eqref{equation}-\eqref{bcondition}]\label{lemma1}
	Let $u\in C(\overline{\Omega})$ be a nonzero convex function and $f(x,z,q):\Omega	\times(-\infty,0) \times \mathbb{R}^n\to\mathbb{R}$ satisfies $(f_2)$. 
\begin{enumerate}[(i)]
\item (comparison principle for viscosity subsolution) If $u$ is a viscosity subsolution to \eqref{equation}  over $\Omega$ and $\phi\in C^2(\Omega)\cap C(\overline{\Omega})$ is convex function satisfying 
\begin{equation*}\begin{split}
		\det D^2 \phi&<f(x,\phi,\nabla \phi)\text{\hspace{0.2cm} in } \Omega,\\
		u&\leqslant\phi\leqslant 0\text{\hspace{0.2cm} on } \partial\Omega,
\end{split}\end{equation*}
then 
\[u\leqslant \phi \text{\hspace{0.2cm} in } \overline{\Omega}.\]
\item (comparison principle for viscosity supersolution) If $u$ is a viscosity supersolution to \eqref{equation}  over $\Omega$ and $\phi\in C^2(\Omega)\cap C(\overline{\Omega})$ is convex function satisfying 
\begin{equation*}\begin{split}
		\det D^2 \phi&>f(x,\phi,\nabla \phi)\text{\hspace{0.2cm} in } \Omega,\\
		0&\geqslant u\geqslant \phi\text{\hspace{0.2cm} on } \partial\Omega,
\end{split}\end{equation*}
then 
\[u\geqslant \phi \text{\hspace{0.2cm} in } \overline{\Omega}.\]
\end{enumerate}
\end{Lemma}
\begin{proof}
We will only give the details of (i), because  (ii) can be proved in a similar manner. 

Now it suffices to prove that $$\displaystyle\min_{x\in\overline{\Omega}}(\phi-u)\geqslant 0.$$
The proof is by contradiction. 
Suppose there exists a point  $x_0\in\overline{\Omega}$ such that
\begin{equation}\label{eq1}
\min_{x\in\overline{\Omega}}(\phi-u)=(\phi-u)(x_0)< 0.
\end{equation}   
Since $(\phi-u)\big|_{\partial\Omega}\geqslant 0$, we obtain $x_0\notin\partial\Omega$, which means $x_0\in\Omega$. 
It immediately follows that $\phi\in C^2(U(x_0)\cap \Omega)$ and 
\begin{equation*} 
(\phi-u)(x)\geqslant (\phi-u)(x_0),\ \ \forall x\in U(x_0)\cap\Omega.
\end{equation*}
Using the condition that $u$ is a viscosity subsolution to \eqref{equation}  over $\Omega$, we then have
\begin{equation}\label{eq2}
\det D^2\phi(x_0)\geqslant    f(x_0,u(x_0),\nabla \phi(x_0)).
\end{equation}
By virtue of $u\big|_{\partial\Omega}\leqslant 0$ and the fact that $u$ is a nonzero convex function, there holds $u(x_0)<0$. Together with \eqref{eq1}, it now yields $\phi(x_0)<u(x_0)<0$. According to \eqref{eq2} and $(f_2)$, we derive
\[\det D^2\phi(x_0)\geqslant    f(x_0,\phi(x_0),\nabla \phi(x_0)),\]
which contradicts the condition that $\det D^2 \phi<f(x,\phi,\nabla \phi)$ in  $\Omega$. We have thus proved the lemma.
\end{proof}

\section{Construction of function $W$}\label{Sec:W}

In this section, we first construct an auxiliary function $W,$  and then show that $W$ is a viscosity subsolution to the equation \eqref{equation}. This is motivated by the method of establishing \textit{a priori} global H\"older estimates  in \cite{Li-Li}.

\begin{Lemma}\label{lemma2}
Let $\Omega\subset\mathbb{R}^n$ be a bounded convex domain. Then for any $x_0\in\partial\Omega$, there exists a convex function $W\in C^2(\Omega)\cap C(\overline{\Omega})$ such that $W(x_0)=0$, $W\big|_{\partial\Omega}\leqslant 0$, $W$ is  a classical subsolution to \eqref{equation}  over $\Omega$ and also a  viscosity subsolution to \eqref{equation}  over $\Omega$.
\end{Lemma}

\begin{proof}
By some translations and rotations, we can assume  $\overline{\Omega}\subset\mathbb{R}^n_+$ and $x_0=0$.

We denote $x=(x',x_n)$, $x'=(x_1,...,x_{n-1})$ and
$r=\left|x^{\prime}\right|=\sqrt{x_{1}^{2}+\cdots + x_{n-1}^{2}}$. 
Let $l=\operatorname{diam}(\Omega)$ and  
\begin{equation}\label{W}
W(x)=	W(r, x_{n})=-Mx_n^\lambda\cdot(N^2l^2-r^2)^{\frac{1}{2}},
\end{equation}
where $\lambda\in(0,1)$, $M>1$, $N>1$ are  constants to be determined later. 

By direct calculation, we obtain 
\begin{align*}
	W_r&=Mrx_n^\lambda\cdot(N^2l^2-r^2)^{-\frac{1}{2}},\\
	W_n&=-\lambda M x_n^{\lambda-1}\cdot (N^2l^2-r^2)^{\frac{1}{2}},\\
	W_{rr}&=MN^2l^2x_n^\lambda\cdot (N^2l^2-r^2)^{-\frac{3}{2}},\\
	W_{nn}&=\lambda(1-\lambda) M x_n^{\lambda-2}\cdot (N^2l^2-r^2)^{\frac{1}{2}},\\
	W_{rn}&=\lambda M rx_n^{\lambda-1}\cdot(N^2l^2-r^2)^{-\frac{1}{2}}.
\end{align*}
It follows that 
\begin{equation}\label{eq7}
W_{rr}\cdot W_{nn}-W_{rn}^2=\lambda M^2N^2l^2 x_n^{2\lambda-2}\cdot\big(1-(1+r^2N^{-2}l^{-2})\lambda\big) \cdot(N^2l^2-r^2)^{-1},
\end{equation}
and 
\begin{equation}\label{eq3}
|\nabla W|=\sqrt{W_r^2+W_n^2}=\sqrt{\lambda^2+\frac{r^2x_n^2}{(N^2l^2-r^2)^2}}\cdot M x_n^{\lambda-1}\cdot(N^2l^2-r^2)^{\frac{1}{2}}.
\end{equation}

Using \eqref{eq7} and the explicit formula for $\det D^2W$ (see Lemma 2.9 in \cite{Li-Li}), we have
\begin{align*}
\det D^2W
&=\left(\frac{W_r}{r}\right)^{n-2}(W_{rr}\cdot W_{nn}-W_{rn}^2)\\
&=(Mx_n^\lambda\cdot(N^2l^2-r^2)^{-\frac{1}{2}})^{n-2}\cdot \lambda M^2N^2l^2 x_n^{2\lambda-2}\cdot\big(1-(1+r^2N^{-2}l^{-2})\lambda\big) \cdot(N^2l^2-r^2)^{-1}\\
&=\lambda M^nN^2l^2 x_n^{n\lambda-2}\cdot\big(1-(1+r^2N^{-2}l^{-2})\lambda\big)\cdot(N^2l^2-r^2)^{-\frac{n}{2}}.
\stepcounter{equation}\tag{\theequation}\label{eq8}
\end{align*}

Now it is time to estimate $[f(x,W,\nabla W)]^{-1}$. 
By \eqref{W} and \eqref{eq3}, we infer that
\begin{equation*}
	|\nabla W|=\sqrt{\lambda^2+\frac{r^2x_n^2}{(N^2l^2-r^2)^2}}\cdot\frac{|W|}{x_n}.
\end{equation*}
For any $x\in\Omega\subset \mathbb{R}_+^n$, we have 
\begin{equation}\label{r and xn}
r\in [0,l],\ \  x_n\in[0,l].
\end{equation} 
We can derive
\begin{equation*}
0\leqslant\frac{r^2x_n^2}{(N^2l^2-r^2)^2}\leqslant \frac{l^2\cdot l^2}{(N^2l^2-l^2)^2}=\frac{1}{(N^2-1)^2},
\end{equation*}
and hence
\begin{equation*}
\sqrt{\lambda^2+\frac{r^2x_n^2}{(N^2l^2-r^2)^2}}\in\left[\lambda,\sqrt{\lambda^2+\frac{1}{(N^2-1)^2}}\ \right].
\end{equation*}
Consequently, there holds
\begin{equation}\label{eq4}
	|\nabla W|\in\left[\lambda\cdot\frac{|W|}{x_n},\sqrt{\lambda^2+\frac{1}{(N^2-1)^2}}\cdot \frac{|W|}{x_n} \right].
\end{equation}
Since $\lambda\in(0,1)$ gives $\lambda-1<0$, then by \eqref{eq3}, we also have
\begin{equation*}
|\nabla W|\geqslant \lambda M l^{\lambda-1}\cdot(N^2l^2-l^2)^{\frac{1}{2}}=\lambda Ml^\lambda \cdot(N^2-1)^{\frac{1}{2}},
\end{equation*}
which immediately gives
\begin{equation*}
1\leqslant\left(\frac{|\nabla W|}{\lambda Ml^\lambda \cdot(N^2-1)^{\frac{1}{2}}}\right)^2=\frac{1}{\lambda^2M^2l^{2\lambda}\cdot(N^2-1)}\cdot|\nabla W|^2.
\end{equation*}
Therefore we derive
\begin{equation}\label{eq5}
1+|\nabla W|^2\in \left[|\nabla W|^2,\Big(\frac{1}{\lambda^2M^2l^{2\lambda}\cdot(N^2-1)}+1\Big)\cdot|\nabla W|^2\right].
\end{equation}
Putting \eqref{eq4} into \eqref{eq5} leads us to
\begin{equation*}
	1+|\nabla W|^2\in \left[\lambda^2\cdot\frac{|W|^2}{x_n^2},\Big(\frac{1}{\lambda^2M^2l^{2\lambda}\cdot(N^2-1)}+1\Big)\cdot\Big(\lambda^2+\frac{1}{(N^2-1)^2}\Big)\cdot\frac{|W|^2}{x_n^2}\right].
\end{equation*}
Then we have
\begin{equation*}
(1+|\nabla W|^2)^{-\frac{\gamma}{2}}\geqslant
\begin{cases}
\displaystyle\Big(\frac{1}{\lambda^2M^2l^{2\lambda}\cdot(N^2-1)}+1\Big)^{-\frac{\gamma}{2}}\cdot\Big(\lambda^2+\frac{1}{(N^2-1)^2}\Big)^{-\frac{\gamma}{2}}\cdot\Big(\frac{|W|}{x_n}\Big)^{-\gamma} &\text{ if }\gamma\geqslant 0,\\
\displaystyle\lambda^{-\gamma}\cdot\Big(\frac{|W|}{x_n}\Big)^{-\gamma}&\text{ if }\gamma< 0.
\end{cases}
\end{equation*}
It follows that  for any $\gamma\in\mathbb{R}$, there holds 
\begin{equation}\label{eq6}
(1+|\nabla W|^2)^{-\frac{\gamma}{2}}\geqslant C_1(\lambda, M,N,l,\gamma)\cdot\Big(\frac{|W|}{x_n}\Big)^{-\gamma}.
\end{equation}
Here, 
\begin{equation*}
C_1(\lambda, M,N,l,\gamma)=\min\left\{\Big(\frac{1}{\lambda^2M^2l^{2\lambda}\cdot(N^2-1)}+1\Big)^{-\frac{\gamma}{2}}\cdot\Big(\lambda^2+\frac{1}{(N^2-1)^2}\Big)^{-\frac{\gamma}{2}},\lambda^{-\gamma}\right\}>0,
\end{equation*}
which satisfies
\begin{equation*}
\lim_{M\to+\infty}C_1(\lambda, M,N,l,\gamma)=\min\left\{\Big(\lambda^2+\frac{1}{(N^2-1)^2}\Big)^{-\frac{\gamma}{2}},\lambda^{-\gamma}\right\}>0.
\end{equation*} 
In view of \eqref{eq6},  $\beta\geqslant n+1$ in the structure condition  $(f_3)$ and $d_x\leqslant x_n$, we can deduce from $(f_3)$ that
\begin{align*}
	[f(x,W,\nabla W)]^{-1}&\geqslant 
	A^{-1} d_x{}^{n+1-\beta}|W|^{\alpha}(1+|\nabla W|^2)^{-\frac{\gamma}{2}}\\
	&\geqslant A^{-1} x_n{}^{n+1-\beta}|W|^{\alpha}\cdot
	C_1(\lambda, M,N,l,\gamma)\cdot\Big(\frac{|W|}{x_n}\Big)^{-\gamma}\\
	&=A^{-1}C_1(\lambda, M,N,l,\gamma) x_n{}^{n+1-\beta+\gamma}|W|^{\alpha-\gamma}\\
	&=A^{-1}C_1(\lambda, M,N,l,\gamma) x_n{}^{n+1-\beta+\gamma}\big(Mx_n^\lambda\cdot(N^2l^2-r^2)^{\frac{1}{2}}\big)^{\alpha-\gamma}\\
	&=A^{-1}C_1(\lambda, M,N,l,\gamma)M^{\alpha-\gamma}x_n{}^{n+1-\beta+\gamma+\lambda(\alpha-\gamma)}\cdot(N^2l^2-r^2)^{\frac{\alpha-\gamma}{2}}.\stepcounter{equation}\tag{\theequation}\label{eq9}
\end{align*}

Combining \eqref{eq8} and \eqref{eq9} gives rise to 
\begin{align*}
	&\ \ \ \ \det D^2W\cdot[f(x,W,\nabla W)]^{-1}\\
	&\geqslant \lambda M^nN^2l^2 x_n^{n\lambda-2}\cdot\big(1-(1+r^2N^{-2}l^{-2})\lambda\big)\cdot(N^2l^2-r^2)^{-\frac{n}{2}}\\
	&\ \ \ \ \cdot A^{-1}C_1(\lambda, M,N,l,\gamma)M^{\alpha-\gamma}x_n{}^{n+1-\beta+\gamma+\lambda(\alpha-\gamma)}\cdot(N^2l^2-r^2)^{\frac{\alpha-\gamma}{2}}\\
	&=A^{-1}C_1(\lambda, M,N,l,\gamma)\lambda  M^{n+\alpha-\gamma}N^2l^2\cdot x_n{}^{\lambda(n+\alpha-\gamma)+n-1-\beta+\gamma} \cdot\big(1-(1+r^2N^{-2}l^{-2})\lambda\big)\\
	&\ \ \ \ \cdot(N^2l^2-r^2)^{\frac{\alpha-\gamma-n}{2}}.\stepcounter{equation}\tag{\theequation}\label{eq13}
\end{align*}

By virtue of $\gamma<\min\{n+\alpha,\beta-n+1\}$ in the structure condition $(f_3)$, we note that $n+\alpha>\gamma$ gives $n+\alpha-\gamma>0$ and $\beta-n+1>\gamma$ gives $n-1-\beta+\gamma<0$. Hence we have 
\[\lim_{\lambda\to 0^+}\left(\lambda(n+\alpha-\gamma)+n-1-\beta+\gamma\right)=n-1-\beta+\gamma<0.\]
This means that there exists a constant  $\lambda_0=\lambda_0(n,\alpha,\beta,\gamma)\in (0,1)$ such that 
\[\lambda_0(n+\alpha-\gamma)+n-1-\beta+\gamma<0.\]
Now we take $\lambda=\lambda_0$ (here we determine the value of $\lambda$). It then follows that 
\begin{equation}\label{eq10}
x_n{}^{\lambda_0(n+\alpha-\gamma)+n-1-\beta+\gamma}\geqslant l{}^{\lambda_0(n+\alpha-\gamma)+n-1-\beta+\gamma}.
\end{equation}
Moreover, we obtain
\begin{equation*}
1-(1+r^2N^{-2}l^{-2})\lambda_0=\left(1-\lambda_0-\frac{1}{N^2}\big(\frac{r}{l}\big)^2\cdot\lambda_0\right)\geqslant 1-\lambda_0-\frac{1}{N^2}
\end{equation*}
as a result of $\lambda_0<1$ and $\frac{r}{l}<1$. Let us take $N=N_0:=\sqrt{\frac{2}{1-\lambda_0}}$ (here we determine the value of $N$). Then we derive 
\begin{equation}\label{eq11}
	1-(1+r^2N_0^{-2}l^{-2})\lambda_0\geqslant \frac{1-\lambda_0}{2}>0.
\end{equation}
In view of $$N_0^2l^2-r^2\in [(N_0^2-1)l^2,N_0^2l^2], $$
there exists a constant  $C_2>0$ such that 
\begin{equation}\label{eq12}
(N_0^2l^2-r^2)^{\frac{\alpha-\gamma-n}{2}}\geqslant C_2>0.
\end{equation}
Substituting \eqref{eq10}, \eqref{eq11} and \eqref{eq12} back to \eqref{eq13}, we deduce that 
\begin{align*}
	&\ \ \ \ \det D^2W\cdot[f(x,W,\nabla W)]^{-1}\\
	&\geqslant 
	A^{-1}C_1(\lambda_0, M,N_0,l,\gamma)\lambda_0  M^{n+\alpha-\gamma}N_0^2l^2\cdot l{}^{\lambda_0(n+\alpha-\gamma)+n-1-\beta+\gamma} \cdot\frac{1-\lambda_0}{2}\cdot C_2\\
	&=A^{-1}C_1(\lambda_0, M,N_0,l,\gamma)C_2\cdot\frac{\lambda_0(1-\lambda_0)}{2}\cdot M^{n+\alpha-\gamma}N_0^2\cdot l{}^{\lambda_0(n+\alpha-\gamma)+n+1-\beta+\gamma}.
\end{align*}
Noticing 
$$\lim_{M\to+\infty}C_1(\lambda_0, M,N_0,l,\gamma)>0,\ \ \lim_{M\to+\infty}M^{n+\alpha-\gamma}=+\infty,$$
we obtain that 
there exists a constant  $M_0>0$ such that 
\[A^{-1}C_1(\lambda_0, M_0,N_0,l,\gamma)C_2\cdot\frac{\lambda_0(1-\lambda_0)}{2}\cdot M_0^{n+\alpha-\gamma}N_0^2\cdot l{}^{\lambda_0(n+\alpha-\gamma)+n+1-\beta+\gamma}> 1.\]
It implies that taking $M=M_0$ (here we determine the value of $M$) finally yields 
\[\det D^2W\cdot[f(x,W,\nabla W)]^{-1}> 1.\]

Up to now, we have determined $\lambda=\lambda_0$, $N=N_0$ and $M=M_0$, and hence 
\begin{equation}\label{determine W}
W(x)=	W(r, x_{n})=-M_0x_n^{\lambda_0}\cdot(N_0^2l^2-r^2)^{\frac{1}{2}},
\end{equation}
which now clearly satisfies $W\in C^2(\Omega)\cap C(\overline{\Omega})$, $W(0)=0$ and 
\begin{equation*}\begin{split}
		\det D^2 W&> f(x,W,\nabla W)\text{\hspace{0.2cm} in } \Omega,\\
		W&\leqslant  0\text{\hspace{0.2cm} on } \partial\Omega.
\end{split}\end{equation*}
Thus $W$ is a classical subsolution to \eqref{equation}  over $\Omega$, and therefore by Lemma \ref{lemma:equiv}, $W$ is a viscosity subsolution to \eqref{equation}  over $\Omega$. The proof of the lemma is now complete. 
\end{proof}
 
\section{Construction of function $u$}
 \label{Sec:u}

Consider the following set of functions:
\[S=\{v:\ v \text{ is a viscosity subsolution to \eqref{equation}  over $\Omega$ and }v\big|_{\partial\Omega}\leqslant 0\}.\]
By Lemma \ref{lemma2}, there exists a function $W$ such that it is  a viscosity subsolution to \eqref{equation}  over $\Omega$ and $W\big|_{\partial\Omega}\leqslant 0$ (for example  \eqref{determine W} suffices). This means $W\in S$ and hence $S\neq \emptyset$.

Now we define 
\begin{equation}\label{defineu}
u=\sup_{v\in S}v.
\end{equation}
In what follows, we   derive the  properties of $u$. 

\begin{Property}\label{property1} $u\big|_{\partial\Omega}=0$. 
\end{Property}
\begin{proof}
	On one hand, for any $v\in S$, we have $v\big|_{\partial \Omega}\leqslant 0$, which implies $u\big|_{\partial \Omega}\leqslant 0$.
	
	On the other hand, for any $x_0\in \partial \Omega$,  Lemma \ref{lemma2} ensures the existence of $W$ such that $W\in S$ and  $W(x_0)=0$. It follows that $u(x_0)=\sup_{v\in S}v(x_0)\geqslant W(x_0)=0$. Since $x_0$ is arbitrary, we have $u\big|_{\partial\Omega}\geqslant 0$.
	
	The above two aspects lead us to $u\big|_{\partial \Omega}= 0$.
\end{proof}

\begin{Property}\label{property2} $u$ is a convex function over $\overline{\Omega}$. 
\end{Property}
\begin{proof}
For any $x_0,y_0\in\overline{\Omega}$ and $z_0=t_0x_0+(1-t_0)y_0$ where $t_0\in(0,1)$, by \eqref{defineu}, we know $u(z_0)=\sup_{v\in S}v(z_0)$. Then there exists a sequence $\{v_k\}\subset S$ such that $\lim_{k\to\infty}v_k(z_0)=u(z_0)$. Since $v_k$ is convex, there holds
$v_k(z_0)\leqslant t_0v_k(x_0)+(1-t_0)v_k(y_0)$. By definition of supremum, we have 
$v_k(x_0)\leqslant u(x_0)$ and $v_k(y_0)\leqslant u(y_0)$. It follows that 
\begin{align*}
u(z_0)&=\lim_{k\to\infty}v_k(z_0)=\varlimsup_{k\to\infty}v_k(z_0)\\
&\leqslant \varlimsup_{k\to\infty}\left(t_0v_k(x_0)+(1-t_0)v_k(y_0)\right)\\
&\leqslant 
 t_0u(x_0)+(1-t_0)u(y_0),
\end{align*}
which implies that $u$ is a convex function over $\overline{\Omega}$. 
\end{proof}

\begin{Property}\label{property3} $u\in C(\overline{\Omega})$.
\end{Property}
\begin{proof}
By Property \ref{property2} and the openness of $\Omega$, we have $u\in C(\Omega)$. It remains to prove that for any $x_0\in\partial\Omega$, $u$ is continuous at $x_0$. By some translations and rotations, we can assume  $\overline{\Omega}\subset\mathbb{R}^n_+$ and $x_0=0$. Using \eqref{determine W} in Lemma \ref{lemma2}, we obtain that $W(x)=-M_0x_n^{\lambda_0}\cdot(N_0^2l^2-r^2)^{\frac{1}{2}}$ satisfies $W\in S$ and $W(0)=0$. By \eqref{defineu}, there holds $u(x)=\sup_{v\in S}v(x)\geqslant W(x)$. By Property \ref{property1} and Property \ref{property2}, there holds $u(x)\leqslant 0$ for any $x\in\Omega$. Hence we have 
\[0\geqslant u(x)\geqslant W(x)=-M_0x_n^{\lambda_0}\cdot(N_0^2l^2-r^2)^{\frac{1}{2}}.\]
This forces
\[|u(x)-u(0)|=|u(x)|\leqslant M_0x_n^{\lambda_0}\cdot(N_0^2l^2-r^2)^{\frac{1}{2}}\to 0\ \ \text{as }x\to 0, \]
i.e. $u$ is contiunous at $0$. Thus we have proved  $u\in C(\overline{\Omega})$.
\end{proof}

\begin{Property}\label{property4}
	$u<0$ in $\Omega$. 
\end{Property}
\begin{proof}
Take any $B_{r_0}(x_0)\subset\subset \Omega$, where $B_{r_0}(x_0)$ represents a ball with center  $x_0$ and radius $r_0$. Consider 
\[\phi_0(x)=-\frac{r_0^2}{2}+\frac{|x-x_0|^2}{2}.\]
Clearly, 
 $\nabla^2\phi_0=\mathrm{I}_n$ is the $n$-order identity matrix and then $\det D^2\phi_0=1$. For any $x\in\overline{B_{r_0}(x_0)}$ and $\epsilon\in (0,1]$, we note that
\[|\nabla(\epsilon\phi_0)|\leqslant|\nabla\phi_0|\leqslant r_0,\ \ 0\geqslant\epsilon\phi_0\geqslant\phi_0\geqslant-\frac{r_0^2}{2}.\]
By the structure condition $(f_2)$, it follows that 
\begin{equation*}
f(x,\epsilon\phi_0,\nabla(\epsilon\phi_0))\geqslant f(x,-\frac{r_0^2}{2},\nabla(\epsilon\phi_0))\geqslant\min_{(x,z,q)\in K_0}f(x,z,q),
\end{equation*}
where 
\[K_0=\overline{B_{r_0}(x_0)}\times\Big\{-\frac{r_0^2}{2}\Big\}\times\overline{B_{r_0}(0)}\subset\Omega\times(-\infty,0)\times\mathbb{R}^n.\]
It is trivial that $K_0$ is compact, and the structure condition  $(f_1)$ gives $f(x,z,q)\in C(K_0)$. By the structure condition $(f_3)$, we have $f(x,z,q)>0$ for any $(x,z,q)\in K_0$. Consequently, there holds $$\eta_0:=\min_{(x,z,q)\in K_0}f(x,z,q)>0.$$
Now we fix $\epsilon_0\in (0,\min\{1,\eta_0^{\frac{1}{n}}\})$. Then  we have
\begin{equation}\label{eq14}
\det D^2(\epsilon_0\phi_0)=\epsilon_0^n\det D^2\phi_0=\epsilon_0^n<\eta_0\leqslant f(x,\epsilon_0\phi_0,\nabla(\epsilon_0\phi_0)),\ \ \forall x\in B_{r_0}(x_0).
\end{equation}

For any $v\in S$, we have $v\big|_{\partial\Omega}\leqslant 0$ and $v$ is convex over $\Omega$, which implies that $v(x)\leqslant v\big|_{\partial\Omega}\leqslant 0$ for any $x\in\Omega$. 
Together with $B_{r_0}(x_0)\subset\subset \Omega$, this gives $v\big|_{\partial B_{r_0}(x_0)}\leqslant 0$. We also note that $\phi_0\big|_{\partial B_{r_0}(x_0)}=0$. As a result, we have
\begin{equation}\label{eq15}
v\big|_{\partial B_{r_0}(x_0)}\leqslant \epsilon_0\phi_0\big|_{\partial B_{r_0}(x_0)}=0.
\end{equation}

According to  \eqref{eq14}, \eqref{eq15} and Lemma \ref{lemma1} (i) (comparison principle for the viscosity subsolution $v$), we derive 
\[v\leqslant\epsilon_0\phi_0,\ \ \forall x\in \overline{B_{r_0}(x_0)},\]
which yields
\[u=\sup_{v\in S}v\leqslant\epsilon_0\phi_0,\ \ \forall x\in \overline{B_{r_0}(x_0)}.\]
Thus there holds
\[u(x_0)=\sup_{v\in S}v(x_0)\leqslant\epsilon_0\phi_0(x_0)=-\frac{\epsilon_0r_0^2}{2}<0,\]
and therefore
\begin{equation}\label{eq16}
u\not\equiv 0\ \ \text{in }\Omega.
\end{equation}

Together with Property \ref{property1} and Property \ref{property2}, this result implies $u(x)<0$ for any $x\in\Omega$. In fact, if not, then there exists $x_0\in\Omega$ such that $u(x_0)\geqslant0$, and hence using Property \ref{property1} and Property \ref{property2} gives $u(x)\equiv0$ for any $x\in\Omega$, which contradicts \eqref{eq16}. The proof is now complete. 
\end{proof}

\section{Existence of viscosity solution}\label{Sec:EV}
Let us prove  Theorem \ref{thm1} in this section.  
Based on the  four properties in the last section, we will show that $u$ constructed in  \eqref{defineu} is a viscosity solution to \eqref{equation} over $\Omega$, which will immediately yield that $u$ is a viscosity solution to the problem \eqref{equation}-\eqref{bcondition} over $\Omega$. The proof is divided into the following two steps.

\subsection*{Step 1:}
We show that $u$ constructed in  \eqref{defineu} is a viscosity subsolution to \eqref{equation}  over $\Omega$.

The proof is by contradiction. Suppose not, then by Definition \ref{def:vis} (i), there exist a point $x_0\in\Omega$, an open neighborhood $U_0(x_0)$ and a convex function $\overline{\phi_0}\in C^2(U_0(x_0)\cap\Omega)$ 
satisfying
\[(\overline{\phi_0}-u)(x)\geqslant (\overline{\phi_0}-u)(x_0),\ \ \forall x\in U_0(x_0)\cap\Omega\]
such that
\begin{equation}\label{eq17}
\det D^2\overline{\phi_0}(x_0)< f(x_0,u(x_0),\nabla \overline{\phi_0}(x_0)).
\end{equation}

\smallskip

\noindent
\textbf{Step 1-a}: Construction of $\overline{\phi_1}$.

Denote 
\[\overline{\phi_1}(x)=\overline{\phi_0}(x)-\overline{\phi_0}(x_0)+u(x_0),\ \ \forall x\in U_0(x_0)\cap\Omega.\]
It is clear that $\overline{\phi_1}\in C^2(U_0(x_0)\cap\Omega)$  is convex and 
\begin{equation}\label{eq21}
\overline{\phi_1}(x)\geqslant u(x),\ \ \forall x\in U_0(x_0)\cap\Omega.
\end{equation}
By the definition of $\overline{\phi_1}$ and Property \ref{property4}, 
we also note that 
\[\overline{\phi_1}(x_0)=u(x_0)<0,\ \ \nabla\overline{\phi_1}(x_0)=\nabla\overline{\phi_0}(x_0),\ \ \nabla^2\overline{\phi_1}(x_0)=\nabla^2\overline{\phi_0}(x_0).\]
Then by \eqref{eq17}, 
we obtain 
\begin{equation*}
	\det D^2\overline{\phi_1}(x_0)< f(x_0,\overline{\phi_1}(x_0),\nabla \overline{\phi_1}(x_0)),
\end{equation*}
and therefore there exists a constant  $\overline{\zeta_0}>0$ such that 
\begin{equation}\label{eq18}
	\det D^2\overline{\phi_1}(x_0)- f(x_0,\overline{\phi_1}(x_0),\nabla \overline{\phi_1}(x_0))=-\overline{\zeta_0}<0.
\end{equation}

\smallskip

\noindent
\textbf{Step 1-b}: Construction of $\overline{\phi_2}$.

Let
\[\overline{\phi_2}(x)=\overline{\phi_1}(x)+\frac{\overline{\varepsilon}|x-x_0|^2}{2},\ \ \forall x\in U_0(x_0)\cap\Omega,\]
where $\overline{\varepsilon}\in (0,1)$ is a constant to be determined later. 
It follows that $\overline{\phi_2}\in C^2(U_0(x_0)\cap\Omega)$  is convex, 
and 
\[\overline{\phi_2}(x)\geqslant\overline{\phi_1}(x)\geqslant u(x),\ \ \forall x\in U_0(x_0)\cap\Omega.\]
We also notice that 
\[\overline{\phi_2}(x_0)=\overline{\phi_1}(x_0)=u(x_0)<0,\ \ \nabla\overline{\phi_2}(x_0)=\nabla\overline{\phi_1}(x_0),\ \ \nabla^2\overline{\phi_2}(x_0)=\nabla^2\overline{\phi_1}(x_0)+\overline{\varepsilon} \mathrm{I}_n.\]
According to \eqref{eq18}, we then derive
\begin{align*}
&\ \ \ \ \det D^2\overline{\phi_2}(x_0)- f(x_0,\overline{\phi_2}(x_0),\nabla \overline{\phi_2}(x_0))\\
&=\big(\det D^2\overline{\phi_2}(x_0)-\det D^2\overline{\phi_1}(x_0)\big)+\big(\det D^2\overline{\phi_1}(x_0)- f(x_0,\overline{\phi_1}(x_0),\nabla \overline{\phi_1}(x_0))\big)\\
&=\left(\det \big(D^2\overline{\phi_1}(x_0)+\overline{\varepsilon} \mathrm{I}_n\big)-\det D^2\overline{\phi_1}(x_0)\right)-\overline{\zeta_0}\\
&\to-\overline{\zeta_0}\ \ \text{as }\overline{\varepsilon}\to 0. 
\end{align*}
Thus there exists a sufficiently small constant $\overline{\varepsilon_0}\in(0,1)$ such that when taking $\overline{\varepsilon}=\overline{\varepsilon_0}$ (here we determine the value of $\overline{\varepsilon}$), i.e. 
\[\overline{\phi_2}(x)=\overline{\phi_1}(x)+\frac{\overline{\varepsilon_0}|x-x_0|^2}{2},\ \ \forall x\in U_0(x_0)\cap\Omega,\]
 we have 
\begin{equation}\label{eq19}
	\det D^2\overline{\phi_2}(x_0)- f(x_0,\overline{\phi_2}(x_0),\nabla \overline{\phi_2}(x_0))\leqslant-\frac{\overline{\zeta_0}}{2}<0.
\end{equation}

Denote 
\[F(x,z,q,A)=\det A-f(x,z,q),\ \ \forall (x,z,q,A)\in\Omega\times(-\infty,0)\times\mathbb{R}^n\times\mathbb{R}^{n^2},\]
where 
 $A$ is a 
 symmetric matrix.
We recall that 
\[\det A=\det (A_{ij})=\sum_{j_1 j_2 \cdots j_n}(-1)^{\tau(j_1 j_2 \cdots j_n)}A_{1j_1}A_{2j_2}\cdots A_{nj_n}\]
is a multivariable polynomial function of all $n^2$ variables $A_{ij}$ with $i,j\in\{1,2,\cdots,n\}$, 
where $j_1j_2\cdots j_n$ is an arrangement of numbers $1,2,\cdots,n$, and $\tau(j_1 j_2 \cdots j_n)$ is an inverse number of $j_1j_2\cdots j_n$. It is known that the Taylor expansion of any polynomial function is itself, then any polynomial function is analytic and hence continuous. This implies
$\det A\in C(\mathbb{R}^{n^2})$. 
Combined with the structure condition $(f_1)$, it then gives 
\[F(x,z,q,A)\in C(\Omega\times(-\infty,0)\times\mathbb{R}^n\times\mathbb{R}^{n^2}).\]
In particular, $F(x,z,q,A)$ is continuous at $\big(x_0,\overline{\phi_2}(x_0),\nabla\overline{\phi_2}(x_0),\nabla^2\overline{\phi_2}(x_0)\big)$. 
Thus there exists a sufficiently small constant $\overline{\delta_0}\in(0,|u(x_0)|)$ such that when 
\[\|x-x_0\|_\infty\leqslant\overline{\delta_0},\ \ \|z-\overline{\phi_2}(x_0)\|_\infty\leqslant\overline{\delta_0},\ \ \|q-\nabla\overline{\phi_2}(x_0)\|_\infty\leqslant\overline{\delta_0},\ \ \|A-\nabla^2\overline{\phi_2}(x_0)\|_\infty\leqslant\overline{\delta_0},\]
we have that the matrix $A$ is positive definite and 
\[\big|F(x,z,q,A)-F\big(x_0,\overline{\phi_2}(x_0),\nabla\overline{\phi_2}(x_0),\nabla^2\overline{\phi_2}(x_0)\big)\big|<\frac{\overline{\zeta_0}}{4}.\]
Together with \eqref{eq19} and the definition of $F$, this implies
\[F(x,z,q,A)<F\big(x_0,\overline{\phi_2}(x_0),\nabla\overline{\phi_2}(x_0),\nabla^2\overline{\phi_2}(x_0)\big)+\frac{\overline{\zeta_0}}{4}\leqslant-\frac{\overline{\zeta_0}}{2}+\frac{\overline{\zeta_0}}{4}=-\frac{\overline{\zeta_0}}{4},\]
i.e.
\begin{equation}\label{eq20}
\det A-f(x,z,q)\leqslant-\frac{\overline{\zeta_0}}{4}<0.
\end{equation}

By virtue of $\overline{\phi_2}\in C^2(U_0(x_0)\cap\Omega)$, we note that there exists a sufficiently small constant $\overline{\delta_0}'\in(0,\overline{\delta_0}]$ such that when $$\|x-x_0\|_\infty\leqslant\overline{\delta_0}',$$
we have 
\begin{equation}\label{eq22}
\|\overline{\phi_2}(x)-\overline{\phi_2}(x_0)\|_\infty\leqslant\frac{\overline{\delta_0}}{2},\ \ \|\nabla\overline{\phi_2}(x)-\nabla\overline{\phi_2}(x_0)\|_\infty\leqslant\frac{\overline{\delta_0}}{2},\ \ \|\nabla^2\overline{\phi_2}(x)-\nabla^2\overline{\phi_2}(x_0)\|_\infty\leqslant\frac{\overline{\delta_0}}{2}.
\end{equation}
At the same time, we can assume $B(x_0,\overline{\delta_0}')\subset\subset U_0(x_0)\cap\Omega$ (here we can shrink the value of $\overline{\delta_0}'$ if necessary).

\smallskip

\noindent
\textbf{Step 1-c}: Construction of $\overline{\phi_3}$.

We take \[\overline{h_0}=\min\Big\{\frac{\overline{\delta_0}}{2},\frac{\overline{\varepsilon_0}\overline{\delta_0}'^2}{2}\Big\}>0,\]
and denote 
\[\overline{\phi_3}(x)=\overline{\phi_2}(x)-\overline{h_0},\ \ \forall x\in U_0(x_0)\cap\Omega.\]
It is clear that $\overline{\phi_3}\in C^2(U_0(x_0)\cap\Omega)$  is convex. 
For any $x\in  B(x_0,\overline{\delta_0}')$, by virtue of $\|x-x_0\|_\infty\leqslant\overline{\delta_0}'$ and \eqref{eq22}, there holds 
\begin{equation*}
	|\overline{\phi_3}(x)-\overline{\phi_2}(x_0)|=|\overline{\phi_2}(x)-\overline{h_0}-\overline{\phi_2}(x_0)|\leqslant |\overline{\phi_2}(x)-\overline{\phi_2}(x_0)|+\overline{h_0}\leqslant\frac{\overline{\delta_0}}{2}+\frac{\overline{\delta_0}}{2}=\overline{\delta_0}.
\end{equation*}
We also note that 
\[\nabla\overline{\phi_3}=\nabla\overline{\phi_2},\ \ \nabla^2\overline{\phi_3}=\nabla^2\overline{\phi_2}.\]
In this way, for any $x\in  B(x_0,\overline{\delta_0}')$, by virtue of $\|x-x_0\|_\infty\leqslant\overline{\delta_0}'$, $\overline{\delta_0}'\in(0,\overline{\delta_0}]$ and  \eqref{eq22}, we have 
\begin{equation*}
	\|x-x_0\|_\infty\leqslant\overline{\delta_0},\ \ 	\|\overline{\phi_3}(x)-\overline{\phi_2}(x_0)\|_\infty\leqslant\overline{\delta_0},\ \ \|\nabla\overline{\phi_3}(x)-\nabla\overline{\phi_2}(x_0)\|_\infty\leqslant\overline{\delta_0},\ \ \|\nabla^2\overline{\phi_3}(x)-\nabla^2\overline{\phi_2}(x_0)\|_\infty\leqslant \overline{\delta_0},
\end{equation*}
where the last inequality also implies that the Hessian matrix $\nabla^2\overline{\phi_3}$ is bounded and  positive definite  (we can choose $\overline{\delta_0}$ to be smaller if necessary). 
Up to now, we have verified that $\overline{\phi_3}$ satisfies all conditions proposed for $F(x,z,q,A)$ in \textbf{Step 1-b}. 
It then follows from \eqref{eq20} that 
\begin{equation}\label{eq25}
	\det D^2\overline{\phi_3}-f(x,\overline{\phi_3},\nabla \overline{\phi_3})\leqslant-\frac{\overline{\zeta_0}}{4}<0,\ \ \forall x\in  B(x_0,\overline{\delta_0}').
\end{equation}

Combining  \eqref{eq25} and Definition \ref{def:class} (ii), we obtain that $\overline{\phi_3}$ is a classical supersolution over $B(x_0,\overline{\delta_0}')$.

\smallskip

\noindent
\textbf{Step 1-d}: Relation of $\overline{\phi_3}$ and $u$. 

Consider 
\[U_1=\{x\in U_0(x_0)\cap\Omega:\ \overline{\phi_3}(x)<u(x)\}.\]
For any $x\in U_1$, by the definition of $\overline{\phi_3}(x)$, we infer that
\[\frac{\overline{\varepsilon_0}|x-x_0|^2}{2}<\overline{h_0}+(u(x)-\overline{\phi_1}(x)).\]
Using \eqref{eq21} 
and the fact that  $U_1\subset U_0(x_0)\cap\Omega$, we have $\frac{\overline{\varepsilon_0}}{2}|x-x_0|^2<\overline{h_0}$
 and therefore 
 \begin{equation*}
 	\|x-x_0\|_\infty<\Big(\frac{2}{\overline{\varepsilon_0}}\overline{h_0}\Big)^{\frac{1}{2}}\leqslant\Big(\frac{2}{\overline{\varepsilon_0}}\cdot\frac{\overline{\varepsilon_0}\overline{\delta_0}'^2}{2}\Big)^{\frac{1}{2}}=\overline{\delta_0}',
 \end{equation*}
 which implies $x\in B(x_0,\overline{\delta_0}')$. Thus there holds
 \[U_1\subset B(x_0,\overline{\delta_0}')\subset\subset U_0(x_0)\cap\Omega.\]
 Since $B(x_0,\overline{\delta_0}')$ is an open ball, then for any $x\in \partial B(x_0,\overline{\delta_0}')$, we have $x\notin B(x_0,\overline{\delta_0}')$, which yields $x\notin U_1$ and further  $\overline{\phi_3}(x)\geqslant u(x)$. As a consequence, we get
 \begin{equation}\label{eq23}
 	u\big|_{\partial B(x_0,\overline{\delta_0}')}\leqslant\overline{\phi_3}\big|_{\partial B(x_0,\overline{\delta_0}')}.
 \end{equation} 
  
For any $x\in \partial B(x_0,\overline{\delta_0}')$, using  $\|x-x_0\|_\infty\leqslant\overline{\delta_0}'$ and  \eqref{eq22}, we have 
$\|\overline{\phi_2}(x)-\overline{\phi_2}(x_0)\|_\infty\leqslant\frac{\overline{\delta_0}}{2}$ and thus $\overline{\phi_2}(x)\leqslant\overline{\phi_2}(x_0)+\frac{\overline{\delta_0}}{2}$. By noting $\overline{\phi_2}(x_0)=u(x_0)<0$ and  $\overline{\delta_0}<|u(x_0)|$, we derive 
\[\overline{\phi_2}(x)\leqslant u(x_0)+\frac{|u(x_0)|}{2}=u(x_0)-\frac{u(x_0)}{2}=\frac{u(x_0)}{2}<0,\ \ \forall x\in \partial B(x_0,\overline{\delta_0}').\]
This leads us to 
\[\overline{\phi_2}\big|_{\partial B(x_0,\overline{\delta_0}')}< 0,\]
and hence
\[\overline{\phi_3}\big|_{\partial B(x_0,\overline{\delta_0}')}=(\overline{\phi_2}-\overline{h_0})\big|_{\partial B(x_0,\overline{\delta_0}')}< 0.\]
Together with \eqref{eq23}, this yields
\begin{equation}\label{eq24}
	u\big|_{\partial B(x_0,\overline{\delta_0}')}\leqslant\overline{\phi_3}\big|_{\partial B(x_0,\overline{\delta_0}')}<0.
\end{equation}

\smallskip

\noindent
\textbf{Step 1-e}: Leading to  contradiction. 

We notice that for any $v\in S$, $v$ is also a viscosity subsolution to \eqref{equation} over $B(x_0,\overline{\delta_0}')$, and moreover by \eqref{defineu} and \eqref{eq24}, $v$ satisfies 
\begin{equation}\label{eq26}
v\big|_{\partial B(x_0,\overline{\delta_0}')}\leqslant	u\big|_{\partial B(x_0,\overline{\delta_0}')}\leqslant\overline{\phi_3}\big|_{\partial B(x_0,\overline{\delta_0}')}<0.
\end{equation}
According to \eqref{eq25}, \eqref{eq26} and Lemma \ref{lemma1} (i) (comparison principle for the viscosity subsolution $v$), we acquire 
\begin{equation*}
v\leqslant \overline{\phi_3},\ \ \forall x\in\overline{B(x_0,\overline{\delta_0}')}.
\end{equation*}
Therefore by \eqref{defineu}, we obtain 
\begin{equation*}
	u=\sup_{v\in S}v\leqslant \overline{\phi_3},\ \ \forall x\in\overline{B(x_0,\overline{\delta_0}')}.
\end{equation*}
In particular, we take $x=x_0$ to get 
\begin{equation*}
u(x_0)\leqslant\overline{\phi_3}(x_0)=\overline{\phi_2}(x_0)-\overline{h_0}=u(x_0)-\overline{h_0}.
\end{equation*}
This is a contradiction to $\overline{h_0}>0$. 

Thus we arrive at the conclusion that $u$ constructed in  \eqref{defineu} is a viscosity subsolution to \eqref{equation}  over $\Omega$, which completes the \textbf{Step 1}.

\subsection*{Step 2:} We show that $u$ constructed in  \eqref{defineu} is a viscosity supersolution to \eqref{equation}  over $\Omega$.

The proof is by contradiction. Suppose not, then by Definition \ref{def:vis} (ii), there exist a point $x_0\in\Omega$, an open neighborhood $U_0(x_0)$ and a convex function $\underline{\phi_0}\in C^2(U_0(x_0)\cap\Omega)$ 
satisfying
\[(\underline{\phi_0}-u)(x)\leqslant (\underline{\phi_0}-u)(x_0),\ \ \forall x\in U_0(x_0)\cap\Omega\]
such that
\begin{equation}\label{eq27}
	\det D^2\underline{\phi_0}(x_0)> f(x_0,u(x_0),\nabla \underline{\phi_0}(x_0)).
\end{equation}

\smallskip

\noindent
\textbf{Step 2-a}: Construction of $\underline{\phi_1}$.

Denote 
\[\underline{\phi_1}(x)=\underline{\phi_0}(x)-\underline{\phi_0}(x_0)+u(x_0),\ \ \forall x\in U_0(x_0)\cap\Omega.\]

On one hand, based on
the fact that $\underline{\phi_0}\in C^2(U_0(x_0)\cap\Omega)$ is convex, we note that $\underline{\phi_1}\in C^2(U_0(x_0)\cap\Omega)$ is convex and hence the Hessian matrix $D^2\underline{\phi_1}(x_0)$ is bounded and positive semi-definite. It further implies that all $n$ eigenvalues of  $D^2\underline{\phi_1}(x_0)$, denoted as $\underline{\lambda_1},\underline{\lambda_2},\cdots,\underline{\lambda_n}$, are bounded,  i.e. there exists a constant $M_0>0$ such that
\begin{equation}\label{eq28}
	0\leqslant\underline{\lambda_1},\underline{\lambda_2},\cdots,\underline{\lambda_n}\leqslant M_0.
\end{equation}
Then all $n$ eigenvalues of the matrix $D^2\underline{\phi_1}(x_0)-M_0\mathrm{I}_n$ are  $\underline{\lambda_1}-M_0,\underline{\lambda_2}-M_0,\cdots,\underline{\lambda_n}-M_0$ and satisfy
\[\underline{\lambda_1}-M_0\leqslant 0,\ \ \underline{\lambda_2}-M_0\leqslant 0,\ \ \cdots,\ \ \underline{\lambda_n}-M_0\leqslant 0.\]
This means that the matrix $D^2\underline{\phi_1}(x_0)-M_0\mathrm{I}_n$ is negative semi-definite and therefore
\begin{equation}\label{eq39}
	D^2\underline{\phi_1}(x_0)\leqslant M_0\mathrm{I}_n.
\end{equation}

On the other hand, by the definition of $\underline{\phi_1}$, 
it immediately holds that  
\begin{equation}\label{eq36}
	\underline{\phi_1}(x)\leqslant u(x),\ \ \forall x\in U_0(x_0)\cap\Omega.
\end{equation}
By the definition of $\underline{\phi_1}$ and Property \ref{property4}, 
we also note that 
\[\underline{\phi_1}(x_0)=u(x_0)<0,\ \ \nabla\underline{\phi_1}(x_0)=\nabla\underline{\phi_0}(x_0),\ \ \nabla^2\underline{\phi_1}(x_0)=\nabla^2\underline{\phi_0}(x_0).\]
Then by \eqref{eq27}, 
we obtain 
\begin{equation*}
	\det D^2\underline{\phi_1}(x_0)> f(x_0,\underline{\phi_1}(x_0),\nabla \underline{\phi_1}(x_0)),
\end{equation*}
which means that there exists a constant  $\underline{\zeta_0}>0$ such that 
\begin{equation}\label{eq29}
	\det D^2\underline{\phi_1}(x_0)- f(x_0,\underline{\phi_1}(x_0),\nabla \underline{\phi_1}(x_0))=\underline{\zeta_0}>0.
\end{equation}
Thus using the structure condition $(f_3)$, we have 
\begin{equation*}
	\det D^2\underline{\phi_1}(x_0)= f(x_0,\underline{\phi_1}(x_0),\nabla \underline{\phi_1}(x_0))+\underline{\zeta_0}\geqslant\underline{\zeta_0}>0.
\end{equation*}
It follows that 
\[\underline{\lambda_1}\,\underline{\lambda_2}\cdots\underline{\lambda_n}\geqslant\underline{\zeta_0}.\]
Combined with \eqref{eq28}, this yields for any $j\in\{1,2,\cdots,n\}$ that
\[\underline{\lambda_j}=\frac{\underline{\lambda_1}\,\underline{\lambda_2}\cdots\underline{\lambda_n}}{\underline{\lambda_1}\,\underline{\lambda_2}\cdots\underline{\lambda_{j-1}}\,\widehat{\underline{\lambda_j}}\,\underline{\lambda_{j+1}}\cdots\underline{\lambda_n}}\geqslant\frac{\underline{\zeta_0}}{M_0^{n-1}}>0,\]
which implies that 
there exists a constant $m_0>0$ (for example  $m_0=\frac{\underline{\zeta_0}}{M_0^{n-1}}$ suffices) such that 
\[\underline{\lambda_1},\underline{\lambda_2},\cdots,\underline{\lambda_n}\geqslant m_0.\]
Then all $n$ eigenvalues of the matrix $D^2\underline{\phi_1}-m_0\mathrm{I}_n$ are  $\underline{\lambda_1}-m_0,\underline{\lambda_2}-m_0,\cdots,\underline{\lambda_n}-m_0$ and satisfy
\[\underline{\lambda_1}-m_0\geqslant 0,\ \ \underline{\lambda_2}-m_0\geqslant 0,\ \ \cdots,\ \ \underline{\lambda_n}-m_0\geqslant 0.\]
This means that the matrix $D^2\underline{\phi_1}(x_0)-m_0\mathrm{I}_n$ is positive semi-definite and therefore
\begin{equation}\label{eq30}
	D^2\underline{\phi_1}(x_0)\geqslant m_0\mathrm{I}_n.
\end{equation}

\smallskip

\noindent
\textbf{Step 2-b}: Construction of $\underline{\phi_2}$.

Let
\[\underline{\phi_2}(x)=\underline{\phi_1}(x)-\frac{\underline{\varepsilon}|x-x_0|^2}{2},\ \ \forall x\in U_0(x_0)\cap\Omega,\]
where $\underline{\varepsilon}\in (0,\frac{m_0}{2})$ is a constant to be determined later. Clearly, $\underline{\phi_2}\in C^2(U_0(x_0)\cap\Omega)$ 
 satisfies
\[\underline{\phi_2}(x_0)=\underline{\phi_1}(x_0)=u(x_0)<0,\ \ \nabla\underline{\phi_2}(x_0)=\nabla\underline{\phi_1}(x_0),\ \ \nabla^2\underline{\phi_2}(x_0)=\nabla^2\underline{\phi_1}(x_0)-\underline{\varepsilon} \mathrm{I}_n.\]
Then we also note that
\[\lim_{\underline{\varepsilon}\to 0}\det D^2\underline{\phi_2}(x_0)=\det D^2\underline{\phi_1}(x_0),\]
and 
\[f(x_0,\underline{\phi_2}(x_0),\nabla \underline{\phi_2}(x_0))=f(x_0,\underline{\phi_1}(x_0),\nabla \underline{\phi_1}(x_0)).\]
As a result, based on \eqref{eq39}, \eqref{eq29} and \eqref{eq30}, there exists a sufficiently small constant $\underline{\varepsilon_0}\in(0,\frac{m_0}{2})$ such that when taking $\underline{\varepsilon}=\underline{\varepsilon_0}$ (here we determine the value of $\underline{\varepsilon}$), i.e. 
\[\underline{\phi_2}(x)=\underline{\phi_1}(x)-\frac{\underline{\varepsilon_0}|x-x_0|^2}{2},\ \ \forall x\in U_0(x_0)\cap\Omega,\]
we have 
\begin{equation}\label{eq31}
	\det D^2\underline{\phi_2}(x_0)- f(x_0,\underline{\phi_2}(x_0),\nabla \underline{\phi_2}(x_0))\geqslant\frac{\underline{\zeta_0}}{2}>0,
\end{equation}
and 
\begin{equation}\label{eq32}
\frac{m_0}{2}\mathrm{I}_n\leqslant	D^2\underline{\phi_2}(x_0)=D^2\underline{\phi_1}(x_0)-\varepsilon_0\mathrm{I}_n\leqslant M_0\mathrm{I}_n.
\end{equation}
Here, \eqref{eq32} shows that the Hessian matrix $\nabla^2\underline{\phi_2}(x_0)$ is bounded and positive definite. 

As in \textbf{Step 1-b}, we turn to study 
\[F(x,z,q,A)=\det A-f(x,z,q),\ \ \forall (x,z,q,A)\in\Omega\times(-\infty,0)\times\mathbb{R}^n\times\mathbb{R}^{n^2},\]
where $A$ is a symmetric matrix. 
Similar to above, we can also show that  $F(x,z,q,A)$ is continuous at $\big(x_0,\underline{\phi_2}(x_0),\nabla\underline{\phi_2}(x_0),\nabla^2\underline{\phi_2}(x_0)\big)$. 
Then, by \eqref{eq31} and the same argument used to derive \eqref{eq20},
 there exists a sufficiently small constant $\underline{\delta_0}\in(0,\min\{|u(x_0)|,\frac{m_0}{2}\})$ such that when 
\[\|x-x_0\|_\infty\leqslant\underline{\delta_0},\ \ \|z-\underline{\phi_2}(x_0)\|_\infty\leqslant\underline{\delta_0},\ \ \|q-\nabla\underline{\phi_2}(x_0)\|_\infty\leqslant\underline{\delta_0},\ \ \|A-\nabla^2\underline{\phi_2}(x_0)\|_\infty\leqslant\underline{\delta_0},\]
we have that  the matrix $A$ is positive definite, and 
\begin{equation}\label{eq33}
	\det A-f(x,z,q)\geqslant\frac{\underline{\zeta_0}}{4}>0.
\end{equation}

Since $\underline{\phi_2}\in C^2(U_0(x_0)\cap\Omega)$, we note that there exists a sufficiently small constant $\underline{\delta_0}'\in(0,\underline{\delta_0}]$ such that when $$\|x-x_0\|_\infty\leqslant\underline{\delta_0}',$$
we have 
\begin{equation}\label{eq34}
	\|\underline{\phi_2}(x)-\underline{\phi_2}(x_0)\|_\infty\leqslant\frac{\underline{\delta_0}}{2},\ \ \|\nabla\underline{\phi_2}(x)-\nabla\underline{\phi_2}(x_0)\|_\infty\leqslant\frac{\underline{\delta_0}}{2},\ \ \|\nabla^2\underline{\phi_2}(x)-\nabla^2\underline{\phi_2}(x_0)\|_\infty\leqslant\frac{\underline{\delta_0}}{2}.
\end{equation}
We can also assume $B(x_0,\underline{\delta_0}')\subset\subset U_0(x_0)\cap\Omega$ (here we can shrink the value of $\underline{\delta_0}'$ if necessary).

\smallskip

\noindent
\textbf{Step 2-c}: Construction of $\underline{\phi_3}$.

We take \[\underline{h_0}=\min\Big\{\frac{\underline{\delta_0}}{2},\frac{\underline{\varepsilon_0}\, \underline{\delta_0}'^2}{2}\Big\}>0,\]
and denote 
\[\underline{\phi_3}(x)=\underline{\phi_2}(x)+\underline{h_0},\ \ \forall x\in U_0(x_0)\cap\Omega.\]
It is trivial to see that $\underline{\phi_3}\in C^2(U_0(x_0)\cap\Omega)$. 
Moreover, by  \eqref{eq32}, we get 
\begin{equation*} 
	\frac{m_0}{2}\mathrm{I}_n\leqslant	D^2\underline{\phi_3}(x_0)=D^2\underline{\phi_2}(x_0)\leqslant M_0\mathrm{I}_n, 
\end{equation*}
which shows that the Hessian matrix $\nabla^2\underline{\phi_3}(x_0)$ is bounded and positive definite.  

For any $x\in  B(x_0,\underline{\delta_0}')$, on account of $\|x-x_0\|_\infty\leqslant\underline{\delta_0}'$ and \eqref{eq34}, there holds 
\begin{equation*}
	|\underline{\phi_3}(x)-\underline{\phi_2}(x_0)|=|\underline{\phi_2}(x)+\underline{h_0}-\underline{\phi_2}(x_0)|\leqslant |\underline{\phi_2}(x)-\underline{\phi_2}(x_0)|+\underline{h_0}\leqslant\frac{\underline{\delta_0}}{2}+\frac{\underline{\delta_0}}{2}=\underline{\delta_0}.
\end{equation*}
We also observe that 
\[\nabla\underline{\phi_3}=\nabla\underline{\phi_2},\ \ \nabla^2\underline{\phi_3}=\nabla^2\underline{\phi_2}.\]
Therefore, for any $x\in  B(x_0,\underline{\delta_0}')$, in view of $\|x-x_0\|_\infty\leqslant\underline{\delta_0}'$, $\underline{\delta_0}'\in(0,\underline{\delta_0}]$ and  \eqref{eq34}, we obtain
\begin{equation*}
	\|x-x_0\|_\infty\leqslant\underline{\delta_0},\ \ 	\|\underline{\phi_3}(x)-\underline{\phi_2}(x_0)\|_\infty\leqslant\underline{\delta_0},\ \ \|\nabla\underline{\phi_3}(x)-\nabla\underline{\phi_2}(x_0)\|_\infty\leqslant\underline{\delta_0},\ \ \|\nabla^2\underline{\phi_3}(x)-\nabla^2\underline{\phi_2}(x_0)\|_\infty\leqslant \underline{\delta_0},
\end{equation*}
where the last inequality also implies that the Hessian matrix $\nabla^2\underline{\phi_3}$ is bounded and  positive definite (we can choose $\underline{\delta_0}$ to be smaller if necessary). 
Until now, we have verified that $\underline{\phi_3}$ satisfies all conditions proposed for $F(x,z,q,A)$ in \textbf{Step 2-b}. 
It then follows from \eqref{eq33} that 
\begin{equation}\label{eq35}
	\det D^2\underline{\phi_3}-f(x,\underline{\phi_3},\nabla \underline{\phi_3})\geqslant\frac{\underline{\zeta_0}}{4}>0,\ \ \forall x\in  B(x_0,\underline{\delta_0}').
\end{equation}

Combining  \eqref{eq25} and Definition \ref{def:class} (i), we obtain that $\underline{\phi_3}$ is a classical subsolution over $B(x_0,\underline{\delta_0}')$.

\smallskip

\noindent
\textbf{Step 2-d}: Construction of $\underline{u}$. 

Let us denote
\[U_2=\{x\in U_0(x_0)\cap\Omega:\ \underline{\phi_3}(x)>u(x)\}.\]
We remark here that there holds  $\underline{\phi_3}=u$ on $\partial U_2$. 
For any $x\in U_2$, by the definition of $\underline{\phi_3}(x)$, we acquire
\[\frac{\underline{\varepsilon_0}|x-x_0|^2}{2}<\underline{h_0}+(\underline{\phi_1}(x)-u(x)).\]
Based on \eqref{eq36} 
and the fact that  $U_2\subset U_0(x_0)\cap\Omega$, we then derive $\frac{\underline{\varepsilon_0}}{2}|x-x_0|^2<\underline{h_0}$
and therefore 
\begin{equation*}
	\|x-x_0\|_\infty<\Big(\frac{2}{\underline{\varepsilon_0}}\underline{h_0}\Big)^{\frac{1}{2}}\leqslant\Big(\frac{2}{\underline{\varepsilon_0}}\cdot\frac{\underline{\varepsilon_0}\underline{\delta_0}'^2}{2}\Big)^{\frac{1}{2}}=\underline{\delta_0}',
\end{equation*}
which implies $x\in B(x_0,\underline{\delta_0}')$. Thus there holds
\begin{equation}\label{eq37}
U_2\subset B(x_0,\underline{\delta_0}')\subset\subset U_0(x_0)\cap\Omega.
\end{equation}

Now we consider
\begin{equation*}
\underline{u}=\begin{cases}
\underline{\phi_3}\ \ &\text{ if }x\in U_2,\\
u,\ \ &\text{ if }x\in\overline{\Omega}\setminus U_2.
\end{cases}
\end{equation*}
Clearly, there holds  $\underline{u}=\underline{\phi_3}=u$ on $\partial U_2$. By the definition of $\underline{u}$ and Property \ref{property1}, there also holds  
\begin{equation}\label{eq38}
\underline{u}\big|_{\partial\Omega}=u\big|_{\partial\Omega}=0.
\end{equation}
We claim that $\underline{u}$ is a viscosity subsolution to \eqref{equation} over $\Omega$. In what follows, the proof of this claim will be finished by proving that $\underline{u}$ is a viscosity subsolution to \eqref{equation} over $U_2$, over $\Omega\setminus \overline{U_2}$ and over $\partial U_2$ respectively. 

For any $x\in U_2$, 
 we have $\underline{u}(x)=\underline{\phi_3}(x)$. Then by \eqref{eq35} and \eqref{eq37}, there holds
\begin{equation*}
	\det D^2\underline{u}-f(x,\underline{u},\nabla \underline{u})\geqslant\frac{\underline{\zeta_0}}{4}>0,\ \ \forall x\in  U_2,
\end{equation*}
which shows that $\underline{u}$ is a classical subsolution to \eqref{equation} over $U_2$. Since  $\underline{\phi_2}\in C^2(U_0(x_0)\cap\Omega)$, then $\underline{\phi_2}\in C^2(U_2)$. Thus Lemma \ref{lemma:equiv} implies that $\underline{u}$ is a viscosity subsolution to \eqref{equation} over $U_2$.

For any $x\in \Omega\setminus \overline{U_2}$, we have $\underline{u}=u$. In \textbf{Step 1}, we have proved
that $u$ is a viscosity subsolution to \eqref{equation} over $\Omega$ and hence over $\Omega\setminus \overline{U_2}$. As a result, $\underline{u}$ is viscosity subsolution to \eqref{equation} over $\Omega\setminus \overline{U_2}$.

It remains to show that $\underline{u}$ is viscosity subsolution to \eqref{equation} over $\partial U_2$. Since \eqref{eq37} gives $U_2\subset\subset U_0(x_0)\cap \Omega$, it follows that $\overline{\partial U_2}=\partial U_2\subset \overline{U_2}\subset U_0(x_0)\cap\Omega$, then 
we also have $\partial U_2\subset\subset U_0(x_0)\cap \Omega$. For simplicity of consideration, 
it is sufficient 
 to  show that  $\underline{u}\big|_{U_0(x_0)\cap \Omega}$  which in fact can be still denoted by $\underline{u}$ as 
\begin{equation*}
	\underline{u}=\begin{cases}
		\underline{\phi_3}\ \ &\text{ if }x\in U_2\\
		u,\ \ &\text{ if }x\in(U_0(x_0)\cap\Omega)\setminus U_2
	\end{cases}
\end{equation*}
is viscosity subsolution to \eqref{equation} over $\partial U_2$. 
By the definition of $U_2$, it is easy to find that 
\[\underline{u}=\max\{\underline{\phi_3},u\},\ \ \forall x\in U_0(x_0)\cap\Omega.\]
Then for any point $y_0\in\partial U_2\subset\subset U_0(x_0)\cap \Omega$, any open neighborhood $U(y_0)\subset U_0(x_0)$ and any convex function $\phi\in C^2(U(y_0)\cap\Omega)$ satisfying 
\[(\phi-\underline{u})(y)\geqslant(\phi-\underline{u})(y_0),\ \ \forall y\in U(y_0)\cap\Omega,\]
there holds 
\[\underline{u}(y)=\max\{\underline{\phi_3}(y),u(y)\}\geqslant u(y),\ \ \forall y\in U(y_0)\cap\Omega.\] 
Since $y_0\in\partial U_2$ leads us to $\underline{u}(y_0)=\underline{\phi_3}(y_0)=u(y_0)$, then we obtain 
\[(\phi-u)(y)\geqslant(\phi-u)(y_0),\ \ \forall y\in U(y_0)\cap\Omega.\]
Consequently, based on the \textbf{Step 1} that  $u$ is  a viscosity subsolution to \eqref{equation} over $\Omega$ and hence over $\partial U_2$, together with Definition \ref{def:vis} (i), we deduce that 
\[\det D^2\phi(y_0)\geqslant f(y_0,u(y_0),\nabla \phi(y_0)).\]
Substituting the fact $\underline{u}(y_0)=u(y_0)$ into this inequality immediately gives rise to 
\[\det D^2\phi(y_0)\geqslant f(y_0,\underline{u}(y_0),\nabla \phi(y_0)).\]
Using Definition \ref{def:vis} (i) again, we finally conclude that 
$\underline{u}$ is  a viscosity subsolution to \eqref{equation} over $\partial U_2$.

Summarizing the above analysis, we have proved the claim that 
 $\underline{u}$ is  a viscosity subsolution to \eqref{equation} over $\Omega$.

\smallskip

\noindent
\textbf{Step 2-e}: Leading to  contradiction. 

From \eqref{eq38} and the claim proved in \textbf{Step 2-d}, it is clear that $\underline{u}\in S$. Hence, by the definition of $u$, we have 
\begin{equation*} 
u(x_0)=\sup_{v\in S}v(x_0)\geqslant\underline{u}(x_0).
\end{equation*}
By the definitions of $\underline{\phi_3}$ and $\underline{h_0}$, we note that 
\[\underline{\phi_3}(x_0)=\underline{\phi_2}(x_0)+h_0=u(x_0)+\underline{h_0}>u(x_0),\]
which together with the definition of $U_2$ forces $x_0\in U_2$. Using the definition of $\underline{u}$, we then obtain 
\begin{equation*} 
\underline{u}(x_0)=\underline{\phi_3}(x_0).
\end{equation*}
Putting the above three inequalities together now gives
\[u(x_0)\geqslant\underline{u}(x_0)=\underline{\phi_3}(x_0)>u(x_0).\]
This is impossible.

Now we can summarize that $u$ constructed in  \eqref{defineu} is a viscosity supersolution to \eqref{equation} over $\Omega$, which completes the \textbf{Step 2}.

Summing up \textbf{Step 1} and \textbf{Step 2}, we conclude that $u$ constructed in  \eqref{defineu} is a viscosity solution to \eqref{equation} over $\Omega$. By virtue of Property \ref{property1}, we further obtain that $u$  constructed in  \eqref{defineu}  is a viscosity solution to the problem \eqref{equation}-\eqref{bcondition} over $\Omega$. 
We have thus completed  the proof of Theorem \ref{thm1}.

\section{Interior regularity of viscosity solution}\label{Sec:IR}
In this section, we give the proof of   Theorem \ref{thm2}.  
From now on, let $u$ be a 
 viscosity solution to the problem  \eqref{equation}-\eqref{bcondition} over $\Omega$. We remark here that we do not require   $u$ to be defined by \eqref{defineu} in this section (though  the uniqueness of $u$ will be proved in the next section). 
Then by Definition \ref{def:vis} and Lemma \ref{lemma1}, it is clear that $u$ also satisfies all the four properties in Section \ref{Sec:u}. 
We are now ready to  show (i) and  (ii) respectively.

\subsection*{Proof of (i):}

Using Lemma \ref{lemma1} (i) and  the auxiliary function in Lemma \ref{lemma2}, we have $u\geqslant W$ on $\overline{\Omega}$. 
By  \eqref{determine W} and \eqref{r and xn}, we then derive 
\begin{equation*}
\min_{x\in\overline{\Omega}}u(x)\geqslant \min_{x\in\overline{\Omega}}W(x)=\min_{x\in\overline{\Omega}}\big(-M_0x_n^{\lambda_0}\cdot(N_0^2l^2-r^2)^{\frac{1}{2}}\big)\geqslant -M_0N_0 l^{\lambda_0+1}.
\end{equation*}
Together with Property \ref{property1} and Property \ref{property4}, this yields
\begin{equation}\label{eq40}
-M_0N_0 l^{\lambda_0+1}\leqslant u(x)\leqslant 0,\ \ \forall x\in\overline{\Omega}. 
\end{equation}

Let us fix an arbitrary set  $\Omega_0\subset\subset\Omega$. By Property \ref{property3} and Property \ref{property4}, there exists  a constant $\eta_0\in (0,M_0N_0l^{\lambda_0+1}]$ such that 
\begin{equation*}
\max_{x\in\overline{\Omega_0}}u(x)=-\eta_0<0.
\end{equation*}
Combined with \eqref{eq40}, this leads us to 
\begin{equation}\label{eq41}
-M_0N_0 l^{\lambda_0+1}\leqslant u(x)\leqslant -\eta_0,\ \ \forall x\in\overline{\Omega_0}. 
\end{equation}

For any $x_0\in\Omega$, let us define the subdifferential of $u$ at $x_0$ as 
\[\nabla u(x_0)=\{q\in\mathbb{R}^n:\, u(x)\geqslant u(x_0)+q\cdot(x-x_0),\, \forall x\in\overline{\Omega}\}.\]
Clearly (see Lemma A.20 in \cite{Figalli}), when $u$ is differentiable at $x_0$, it holds that $\nabla u(x_0)$ is a singleton which merely involves  the gradient of $u$ at $x_0$. 
For any $\Omega_0\subset\subset\Omega$, we now denote $$\nabla u(\overline{\Omega_0})=\{q\in\nabla u(x_0):\, x_0\in\overline{\Omega_0}\}.$$

 We claim that $\nabla u(\overline{\Omega_0})$ is a bounded set. In fact, by definition, for any $q\in\nabla u(\overline{\Omega_0})$, there exists a point $x_0\in\overline{\Omega_0}$ such that $$u(x)\geqslant u(x_0)+q\cdot (x-x_0),\ \ \forall x\in\overline{\Omega}.$$
 We observe that there exists a point $y_0\in\partial\Omega$ such that $(y_0-x_0)\parallel q$ with the same direction.
It follows that 
 \begin{equation*}
 0=u(y_0)\geqslant u(x_0)+q\cdot(y_0-x_0)=u(x_0)+|q|\cdot|y_0-x_0|.
 \end{equation*}
Then by \eqref{eq41}, we infer 
\begin{equation*}
|q|\cdot|y_0-x_0|\leqslant -u(x_0)=|u(x_0)|\leqslant M_0N_0l^{\lambda_0+1}.
\end{equation*}
Since $|y_0-x_0|\geqslant\operatorname{dist}(\overline{\Omega_0},\partial\Omega)$, we have 
\[|q|\leqslant\frac{M_0N_0l^{\lambda_0+1}}{\operatorname{dist}(\overline{\Omega_0},\partial\Omega)}=:R_0,\]
which proves that
$\nabla u(\overline{\Omega_0})\subset B_{R_0}(0)$ is a bounded set.

Now we denote 
\begin{equation*}
K=\overline{\Omega_0}\times[-M_0N_0l^{\lambda_0+1},-\eta_0]\times\overline{B_{R_0}(0)}.
\end{equation*}
From the above analysis, we have proved that 
\begin{equation}\label{eq49}
(x,u,\nabla u)\in K.
\end{equation}
It is also clear that $K$ is compact, and the structure condition  $(f_1)$ shows 
\begin{equation*}
f(x,z,q)\in C(K). 
\end{equation*}
By virtue of the structure condition $(f_3)$, there holds $f(x,z,q)>0$ for any $(x,z,q)\in K$. Therefore, we have 
\begin{align*}
&\Lambda_1:=\min_{(x,z,q)\in K}f(x,z,q)>0,\\
&\Lambda_2:=\max_{(x,z,q)\in K}f(x,z,q)<+\infty.
\end{align*}  
By the above arguments, we obtain that 
\begin{equation}\label{eq42}
\Lambda_1\leqslant f(x,u,\nabla u)\leqslant\Lambda_2,\ \ \forall x\in\overline{\Omega_0}.
\end{equation}

Consider the set  $G_t=\{x\in\Omega:\, u(x)<t\}$, where $t\leqslant 0$. We note that $u\big|_{\partial G_t}=t$. It is also trivial to see that for any $t_1, t_2\in\left(\min_{x\in\overline{\Omega}}u(x),0\right)$ satisfying $ t_1<t_2$, we have 
\[G_{t_1}\subset\subset G_{t_2}\subset\subset\Omega,\ \ \lim_{t\to 0^-}G_t=\bigcup_{t<0}G_t=\Omega. \]
Thus there exists $t_0\in\left(\min_{x\in\overline{\Omega}}u(x),0\right)$ such that 
\begin{equation}\label{eq45}
\Omega_0\subset\subset G_{t_0}\subset\subset\Omega.
\end{equation}
Based on this fact, similar to \eqref{eq42}, we can also show that there exists two constants $\Lambda_1',\Lambda_2'>0$ such that
\begin{equation}\label{eq54}
\Lambda_1'\leqslant f(x,u,\nabla u)\leqslant\Lambda_2',\ \ 
\forall x\in\overline{G_{t_0}}.
\end{equation}  
In consequence, there hold 
\begin{equation}\label{eq43}
\Lambda_1'\leqslant \det D^2u\leqslant\Lambda_2',\ \ 
\forall x\in\overline{G_{t_0}},
\end{equation}
and 
\begin{equation}\label{eq44}
u\big|_{\partial G_{t_0}}=t_0<0.
\end{equation}

By \eqref{eq43}-\eqref{eq44} and the Caffarelli's strict convexity result of convex solutions (for the case when the right-hand-side is bounded away from zero and infinity and moreover the boundary data is $C^{1,\alpha}$ with $\alpha>1-\frac{2}{n}$, see Corollary 4.11 in \cite{Figalli} for example), it follows that $u$ is strictly convex in $G_{t_0}$ and hence also in $\Omega_0$. 

Now we 
take a sufficiently small constant $\varepsilon_0>0$ such that 
\begin{equation}\label{eq46}
\Omega_0\subset\subset G_{t_0-\varepsilon_0}\subset\subset G_{t_0}.
\end{equation}
Then by  the Caffarelli's interior $C^{1,\alpha}$ regularity 
 result of strictly convex solutions (for the case when the right-hand-side is bounded away from zero and infinity, see Corollary 4.21 in \cite{Figalli} for example), there exists some constant $\delta=\delta(n,\Lambda_1',\Lambda_2')>0$ such that $$u\in C^{1,\delta}(\overline{G_{t_0-\varepsilon_0}}).$$ 
 As a result, $\nabla u$ exists everywhere over $\overline{G_{t_0-\varepsilon_0}}$ and $$\nabla u\in C^{0,\delta}(\overline{G_{t_0-\varepsilon_0}}),$$
which forces that $\nabla u(\overline{G_{t_0-\varepsilon_0}})$ is a bounded set.
We now apply  the arguments in the proof of \eqref{eq49} again to derive that there exist two constants $\eta_1\in (0,M_0N_0l^{\lambda_0+1}]$ and $R_1\in (0,+\infty)$ such that  for any $x\in\overline{G_{t_0-\varepsilon_0}}$, we have
\begin{equation}\label{eq50}
(x,u,\nabla u)\in \overline{G_{t_0-\varepsilon_0}}\times[-M_0N_0l^{\lambda_0+1},-\eta_1]\times\overline{B_{R_1}(0)}. 
\end{equation}

By \eqref{eq50} and the structure condition $(f_1)$, we see that 
\begin{equation}\label{eq52}
	f(x,u,\nabla u)\in C\big(\overline{G_{t_0-\varepsilon_0}}\times[-M_0N_0l^{\lambda_0+1},-\eta_1]\times\overline{B_{R_1}(0)}\big). 
\end{equation}
According to \eqref{eq46}, \eqref{eq52} and the Caffarelli's interior $W^{2,p}$ regularity result of strictly convex solutions  (for the case when the right-hand-side is continuous, see Corollary 4.38 in \cite{Figalli} for example), it follows that for any $p\in (1,+\infty)$, we have  
$u\in W^{2,p}(\overline{\Omega_0})$ and hence 
there exists some constant $C_3=C_3(n,\Lambda_1,\Lambda_2,p)>0$ such that 
\begin{equation}\label{eq47}
\|u\|_{W^{2,p}(\overline{\Omega_0})}\leqslant C_3.
\end{equation}
For any $\mu\in(0,1)$, it is clear that there exists a sufficiently large constant  $p_0\in(n,+\infty)$ such that $1-\frac{n}{p_0}>\mu$. In particular, \eqref{eq47} still holds for $p=p_0$. Then by the  Sobolev imbedding theorem (see Theorem 4.12 (Part II) in \cite{Adams-Fournier} for example), we have $u\in W^{2,p_0}(\Omega_0)\hookrightarrow C^{1,\mu}(\overline{\Omega_0})$ and therefore there exists some constant  $C_4=C_4(n,p_0,\mu)$ such that 
\begin{equation}\label{eq48}
	\|u\|_{C^{1,\mu}(\overline{\Omega_0})}\leqslant C_4\|u\|_{W^{2,p_0}(\Omega_0)}\leqslant C_3C_4.
\end{equation}
Thus we can summarize from \eqref{eq47}-\eqref{eq48} that $u\in W^{2,p}(\overline{\Omega_0})$ for any $p\in (1,+\infty)$
as well as 
$u\in C^{1,\mu}(\overline{\Omega_0})$ for any  $\mu\in (0,1)$. This completes the proof of (i).

\subsection*{Proof of (ii):}
To show $u\in C^\infty(\Omega)$, 
it suffices to show that $u$ is $C^\infty$ at any point $x_0\in\Omega$. We observe that for any $x_0\in\Omega$,  there holds $x_0\notin\partial\Omega$ and then $\operatorname{dist}(x_0,\partial\Omega)>0$. In this way, we can always take $\Omega_0\subset\subset\Omega$ such that $x_0\in\Omega_0$. Therefore this problem reduces to showing that 
 $u\in C^\infty(\overline{\Omega_0})$ for any $\Omega_0\subset\subset\Omega$.

Now we fix an arbitrary set  $\Omega_0\subset\subset\Omega$. We also construct the set  $G_{t_0}$ as in (i). 
Since we have proved in  \eqref{eq45} that $G_{t_0}\subset\subset\Omega$, then we can directly apply the result   on $\Omega_0$ in (i) to $G_{t_0}$. It immediately follows that for any  $\mu\in (0,1)$, we have 
\begin{equation}\label{eq51}
u\in C^{1,\mu}(\overline{G_{t_0}}),
\end{equation}
which further implies that
 \begin{equation*}
 \nabla u\in C^{0,\mu}(\overline{G_{t_0}}).
 \end{equation*} 
 This also ensures that $\nabla u(\overline{G_{t_0}})$ is a bounded set. Thus we can 
follow   the arguments in the proof of \eqref{eq49} or \eqref{eq50} again to obtain that there exist two constants $\eta_2\in (0,M_0N_0l^{\lambda_0+1}]$ and $R_2\in (0,+\infty)$ such that  for any $x\in\overline{G_{t_0}}$, we have
 \begin{equation}\label{eq60}
 	(x,u,\nabla u)\in  \overline{G_{t_0}}\times[-M_0N_0l^{\lambda_0+1},-\eta_2]\times\overline{B_{R_2}(0)}.
 \end{equation}
Using  the  condition $f(x,z,q)\in C^\infty(\Omega\times (-\infty,0)\times \mathbb{R}^n)$, we then infer that the composition function $f(x,u(x),\nabla u(x))$ satisfies 
 \begin{equation}\label{eq53}
	f(x,u(x),\nabla u(x))\in C^{0,\mu}\big(\overline{G_{t_0}}\big).
\end{equation}
Since we have proved in  \eqref{eq45} that $\Omega_0\subset\subset G_{t_0}$, then by \eqref{eq54}, \eqref{eq53} and the Caffarelli's interior $C^{2,\alpha}$ regularity result of strictly convex solutions (for the case when the right-hand-side is H\"older continuous and also  bounded away from zero and infinity, see Corollary 4.43 in \cite{Figalli} for example), we obtain 
\begin{equation}\label{eq55}
u\in C^{2,\mu}(\overline{\Omega_0}).
\end{equation}

Up to now,  using \eqref{eq45}, \eqref{eq55} and  the  condition $f(x,z,q)\in C^\infty(\Omega\times (-\infty,0)\times \mathbb{R}^n)$, 
 by the same method as from \eqref{eq51} to \eqref{eq53},  we can also derive that the composition function $f(x,u(x),\nabla u(x))$ satisfies  
\[f(x,u(x),\nabla u(x))\in C^{1,\mu}(\overline{\Omega_0}).\]
Consequently, for any  direction $e\in\mathbb{S}^{n-1}$, we can differentiate \eqref{equation} in the direction $e$ to get 
\begin{equation}\label{eq56}
U^{ij}\partial_{ij}u_e=(f(x,u(x),\nabla u(x)))_e,
\end{equation}
where $U^{ij}=\det D^2u\cdot (D^2u)^{-1}$ is the	adjoint matrix of $D^2u$, and 
\begin{equation}\label{eq57} 
(f(x,u(x),\nabla u(x)))_e\in C^{0,\mu}(\overline{\Omega_0}).
\end{equation}
By virtue of \eqref{eq55} and \eqref{eq43}, we can 
also derive that there exist two  constants  $M_0',m_0'>0$ such that 
\[m_0'\mathrm{I}_n\leqslant D^2u\leqslant M_0'\mathrm{I}_n,\ \ \forall x\in \overline{\Omega_0}.\]
Using again \eqref{eq43} now gives rise to 
\[\frac{\Lambda_1'}{M_0'}\mathrm{I}_n\leqslant U^{ij}\leqslant \frac{\Lambda_2'}{m_0'}\mathrm{I}_n,\ \ \forall x\in \overline{\Omega_0}.\]
Hence the linearized equation \eqref{eq56} is uniformly elliptic. Therefore, in view of  \eqref{eq57}  and the Schauder  regularity result (of linear uniformly elliptic equations with H\"older coefficients, see Theorem A.39 in \cite{Figalli} for example), 
we deduce that 
\[u_e\in C^{2,\mu}(\overline{\Omega_0}).\]
Since the direction $e$ is arbitrary, it follows that 
\begin{equation}\label{eq59}
u\in C^{3,\mu}(\overline{\Omega_0}).
\end{equation}

 By induction, we can always repeat the above arguments from \eqref{eq55} to \eqref{eq59} to upgrade the regularity. Therefore we obtain that for any $m\in\mathbb{N}$, there holds
\[u\in C^{m,\mu}(\overline{\Omega_0}).\]
This shows
\[u\in C^\infty(\overline{\Omega_0}),\]
and thus completes the proof of (ii). 

Until now, the proof of Theorem \ref{thm2} has been finished. 

\section{Uniqueness of viscosity solution}\label{Sec:Unique}
In this section, we will prove  Theorem \ref{thm3}. 
Let $u$ and $\widehat{u}$ be two viscosity solutions to the problem \eqref{equation}-\eqref{bcondition} over $\Omega$. It is clear that $u\big|_{\partial\Omega}=\widehat{u}\big|_{\partial\Omega}=0$. We are going to show that $u=\widehat{u}$ in $\Omega$. 

\subsection*{Step 1} We show that $u$ and $\widehat{u}$ are two classical solutions to the problem \eqref{equation}-\eqref{bcondition} over $\Omega$.

Now we fix an arbitrary set  $\Omega_0\subset\subset\Omega$. We also construct the set  $G_{t_0}$ by \eqref{eq45}. 
As in the proof of Theorem \ref{thm2} (ii), 
by \eqref{eq51} and \eqref{eq60}, there hold $u\in C^{1,\mu}(\overline{G_{t_0}})$ and  for any $x\in\overline{G_{t_0}}$, we have
\begin{equation*} 
	(x,u,\nabla u)\in  \overline{G_{t_0}}\times[-M_0N_0l^{\lambda_0+1},-\eta_2]\times\overline{B_{R_2}(0)}.
\end{equation*}

By virtue of the structure condition $(f_1')$, we obtain that 
\begin{equation}\label{eq61}
	f(x,z,q)\in C^{0,\delta_0}\big(\overline{G_{t_0}}\times[-M_0N_0l^{\lambda_0+1},-\eta_2]\times\overline{B_{R_2}(0)}\big).
\end{equation}
For any $x,\widehat{x}\in \overline{G_{t_0}}$, we observe that 
\begin{align*}
&|x-\widehat{x}|=|x-\widehat{x}|^{1-\mu}|x-\widehat{x}|^\mu\leqslant l^{1-\mu}|x-\widehat{x}|^\mu,\\
&|u(x)-u(\widehat{x})|\leqslant\|\nabla u\|_{L^\infty(\overline{G_{t_0}})}|x-\widehat{x}|\leqslant\|u\|_{C^{1,\mu}(\overline{G_{t_0}})}l^{1-\mu}|x-\widehat{x}|^\mu,\\
&|\nabla u(x)-\nabla u(\widehat{x})|\leqslant\|\nabla u\|_{C^{0,\mu}(\overline{G_{t_0}}) } |x-\widehat{x}|^\mu,
\end{align*}
and therefore by \eqref{eq61}, we derive that 
\begin{align*}
&\ \ \ \ |f(x,u(x),\nabla u(x))-f(\widehat{x},u(\widehat{x}),\nabla u(\widehat{x}))|\\
&\leqslant C(t_0,M_0,N_0,\lambda_0,l,\eta_2,R_2)\big(|x-\widehat{x}|^2+|u(x)-u(\widehat{x})|^2+|\nabla u(x)-\nabla u(\widehat{x})|^2
\big)^{\frac{\delta_0}{2}}\\
&\leqslant C(t_0,M_0,N_0,\lambda_0,l,\eta_2,R_2,\mu)|x-\widehat{x}|^{\mu\delta_0}.
\end{align*}
Here  in the last step we have applied \eqref{eq48} of $\Omega_0$ to $G_{t_0}$ (since $G_{t_0}\subset\subset\Omega$ by  \eqref{eq45}).
As a result, 
it follows that
\begin{equation}\label{eq62}
	\big(f(x,u(x),\nabla u(x))\big)\in C^{0,\mu\delta_0}\big(\overline{G_{t_0}}\big).
\end{equation}

Since we have proved in  \eqref{eq45} that $\Omega_0\subset\subset G_{t_0}$, then by \eqref{eq54}, \eqref{eq62} and the Caffarelli's interior $C^{2,\alpha}$ regularity result as before, 
we also obtain 
\begin{equation*} 
	u\in C^{2,\mu\delta_0}(\overline{\Omega_0}).
\end{equation*}
Since $\Omega_0$ is arbitrary, we infer
\begin{equation*}
	u\in C^{2}(\Omega).
\end{equation*}
Thus 
we conclude that $u$ is a classical solution to the problem \eqref{equation}-\eqref{bcondition} over $\Omega$. 
Similarly, we can also show that $\widehat{u}\in C^{2}(\Omega)$ and hence $\widehat{u}$ is a classical solution to the problem \eqref{equation}-\eqref{bcondition} over $\Omega$.

\subsection*{Step 2} We show that $u=\widehat{u}$ in $\Omega$.

We first claim that
\begin{equation}\label{eq64}
u-\widehat{u}\geqslant 0\  \text{ in }\Omega.
\end{equation} 
The proof of this claim is by contradiction. We note that $(u-\widehat{u})\big|_{\partial\Omega}=0$. 
Suppose   there exists a point $\underline{x}\in\Omega$ such that 
\[\min_{x\in\overline{\Omega}}(u-\widehat{u})=(u-\widehat{u})(\underline{x})<0.\]
Then we have
\[(u-\widehat{u})(\underline{x})<0,\ \ \nabla(u-\widehat{u})(\underline{x})=0,\ \ \nabla^2(u-\widehat{u})(\underline{x})\geqslant0,\]
i.e.
\[  u(\underline{x})<\widehat{u}(\underline{x}), \ \ \nabla u(\underline{x})=\nabla\widehat{u}(\underline{x}),\ \ \nabla^2u(\underline{x})\geqslant\nabla^2\widehat{u}(\underline{x}).\]
Together with \textbf{Step 1} and Definition \ref{def:class}, this yields 
\begin{equation*}
	f(\underline{x},u(\underline{x}),\nabla u(\underline{x}))=\det D^2u(\underline{x})\geqslant\det D^2\widehat{u}(\underline{x})=f(\underline{x},\widehat{u}(\underline{x}),\nabla\widehat{u}(\underline{x}))=f(\underline{x},\widehat{u}(\underline{x}),\nabla u(\underline{x})),
\end{equation*}
which contradicts the fact $u(\underline{x})<\widehat{u}(\underline{x})$ by virtue of the structure condition $(f_2')$. This proves the claim.


Similarly, we can also obtain \begin{equation}\label{eq65}
	u-\widehat{u}\leqslant 0\  \text{ in }\Omega.
\end{equation} 
Combining \eqref{eq64} and \eqref{eq65}, we get $u-\widehat{u}=0$ in $\Omega$, i.e. $u=\widehat{u}$ in $\Omega$.

The proof of Theorem \ref{thm3} is now complete.

\subsection*{Acknowledgement:}
This work was supported by the National Natural Science Foundation of China. 

\subsection*{Disclosure Statement:} There are no relevant financial or non-financial competing interests to report.


\begin{thebibliography}{99}
 \bibitem{Adams-Fournier} Adams, Robert A.; Fournier, John J. F. Sobolev spaces. Second edition. Pure and Applied Mathematics (Amsterdam), 140. Elsevier/Academic Press, Amsterdam, 2003. 
 

\bibitem{Caffarelli-Nirenberg-Spruck}  Caffarelli, L.; Nirenberg, L.; Spruck, J. The Dirichlet problem for nonlinear second-order elliptic equations.  \uppercase\expandafter{\romannumeral1}: Monge-Amp\`ere equation. Comm. Pure Appl. Math. 37 (1984), no. 3, 369–402. 
	
\bibitem{Cheng-Yau77}  Cheng, Shiu Yuen; Yau, Shing Tung. On the regularity of the Monge-Amp\`ere equation  $\det\frac{\partial^2u}{\partial x_i\partial x_j}=F(x,u)$. Comm. Pure Appl. Math. 30 (1977), no. 1, 41–68.
	
\bibitem{Chen-Li-Zhu}	Chen, Shibing; Li, Qi-rui; Zhu, Guangxian. On the $L_p$ Monge-Ampère equation. J. Differential Equations 263 (2017), no. 8, 4997-5011.
	
\bibitem{Chou-Wang} Chou, Kai-Seng; Wang, Xu-Jia. The $L_p$-Minkowski problem and the Minkowski problem in centroaffine geometry. Adv. Math. 205 (2006), no. 1, 33-83.
	
\bibitem{Crandall}	Crandall, Michael G. Viscosity solutions: a primer. Viscosity solutions and applications (Montecatini Terme, 1995), 1–43, Lecture Notes in Math., 1660, Fond. CIME/CIME Found. Subser., Springer, Berlin, 1997. 

\bibitem{Figalli}  Figalli, Alessio. The Monge-Amp\`ere equation and its applications. Zurich Lectures in Advanced Mathematics. European Mathematical Society (EMS), Z\"urich, 2017. 

\bibitem{Guti} Gutiérrez, Cristian E. The Monge-Ampère equation. Second edition. Progress in Nonlinear Differential Equations and their Applications, 89. Birkhäuser/Springer, [Cham], 2016.

 

\bibitem{Guanbo} Guan, Bo. 
	The Dirichlet problem for a class of fully nonlinear elliptic equations.
	Comm. Partial Differential Equations 19 (1994), no. 3-4, 399–416.

\bibitem{Guan-Ma} Guan, Pengfei; Ma, Xi-Nan. Convex solutions of fully nonlinear elliptic equations in classical differential geometry. Geometric evolution equations, 115–127, Contemp. Math., 367, Amer. Math. Soc., Providence, RI, 2005.

\bibitem{Hong-Huang-Wang} Hong, Jiaxing; Huang, Genggeng; Wang, Weiye. Existence of global smooth solutions to Dirichlet problem for degenerate elliptic Monge-Ampere equations. Comm. Partial Differential Equations 36 (2011), no. 4, 635-656.

\bibitem{Jiang-Trudinger-Yang} Jiang, Feida; Trudinger, Neil S.; Yang, Xiao-Ping. On the Dirichlet problem for Monge-Ampère type equations. Calc. Var. Partial Differential Equations 49 (2014), no. 3-4, 1223-1236.


\bibitem{Jian-Lu-Zhu} Jian, Huaiyu; Lu, Jian; Zhu, Guangxian. Mirror symmetric solutions to the centro-affine Minkowski problem. Calc. Var. Partial Differential Equations 55 (2016), no. 2, Art. 41, 22 pp.

\bibitem{J-WXD} Jian, Huaiyu; Wang, Xianduo. Sharp boundary regularity for some degenerate-singular Monge-Amp\`ere equations on $k$-convex domain. arXiv:2303.09890


\bibitem{Jian-Wang1} Jian, Huaiyu; Wang, Xu-Jia. Existence of entire solutions to the Monge-Ampère equation. Amer. J. Math. 136 (2014), no. 4, 1093-1106.

 

\bibitem{Le} Le, Nam Q. 
The eigenvalue problem for the Monge-Ampère operator on general bounded convex domains. 
Ann. Sc. Norm. Super. Pisa Cl. Sci. (5) 18 (2018), no. 4, 1519-1559.


	 
\bibitem{Li-Li} 
Li, Mengni; Li, You. Global regularity for a class of Monge-Ampère type equations. Sci. China Math. 65 (2022), no. 3, 501-516.


\bibitem{Liu-Wang} Liu, Jiakun; Wang, Xu-Jia. Interior a priori estimates for the Monge-Ampère equation. Surveys in differential geometry 2014. Regularity and evolution of nonlinear equations, 151–177, Surv. Differ. Geom., 19, Int. Press, Somerville, MA, 2015.

\bibitem{Loewner-Nirenberg}
Loewner, Charles; Nirenberg, Louis. Partial differential equations invariant under conformal or projective transformations. Contributions to analysis (a collection of papers dedicated to Lipman Bers), pp. 245-272. Academic Press, New York, 1974. 


\bibitem{Lutwak} Lutwak, Erwin.  The Brunn-Minkowski-Firey theory. I. Mixed volumes and the Minkowski problem. J. Differential Geom. 38 (1993), no. 1, 131-150. 


\bibitem{Tian} Tian, Gang. On the existence of solutions of a class of Monge-Ampère equations. A Chinese summary appears in Acta Math. Sinica 32 (1989), no. 4, 576. Acta Math. Sinica (N.S.) 4 (1988), no. 3, 250-265.
	
\bibitem{Trudinger-Urbas83}	
Trudinger, Neil S.; Urbas, John I. E. 
The Dirichlet problem for the equation of prescribed Gauss curvature.
Bull. Austral. Math. Soc. 28 (1983), no. 2, 217-231. 
	
\bibitem{Trudinger-Wang08adv} 
Trudinger, Neil S.; Wang, Xu-Jia. The Monge-Ampère equation and its geometric applications. Handbook of geometric analysis. No. 1, 467–524, Adv. Lect. Math. (ALM), 7, Int. Press, Somerville, MA, 2008.


	
\bibitem{Tso} Tso, Kaising. On a real Monge-Amp\`ere functional. Invent. Math. 101 (1990), no. 2, 425–448.


	
	
\bibitem{Urbas88} Urbas, John I. E. 
Global Hölder estimates for equations of Monge-Ampère type.
Invent. Math. 91 (1988), no. 1, 1-29.
	
	
	
\end{thebibliography}
\end{document}